\documentclass[11pt,oneside,openany,notitlepage]{article}
\usepackage{amsfonts}
\usepackage{graphics}
\usepackage[dvips]{color}
\usepackage[pdftex]{graphicx}
\usepackage{multicol}
\usepackage{times}
\usepackage{amssymb}
\usepackage{amsmath}
\usepackage{flafter}
\usepackage[center]{caption2}
\usepackage{indentfirst,latexsym,bm}
\usepackage{picinpar}
\usepackage[numbers,sort&compress]{natbib}
\usepackage[toc,page,title,titletoc,header]{appendix}
\usepackage{tikz}
\usepackage{pgflibraryarrows}
\usepackage{pgflibrarysnakes}
\usepackage{dsfont}
\usepackage[pdftex]{graphicx}

\usepackage[top=1in, bottom=1in, left=1in, right=1in]{geometry}

\newenvironment{proof}
               {\begin{sloppypar} \noindent{\bf Proof}}
               {\hspace*{\fill} $\square$ \end{sloppypar}}
\newtheorem{atheorem}{\bf \temp}[section]
\newtheorem{thm}[atheorem]{Theorem}
\newtheorem{cor}[atheorem]{Corollay}
\newtheorem{lem}[atheorem]{Lemma}

\newtheorem{de}[atheorem]{Definition}
\newtheorem{rem}[atheorem]{Remark}

\def\bc{\begin{center}}
\def\ec{\end{center}}

\makeatletter
\def\@biblabel#1{#1.\hfill}
\makeatother

\numberwithin{equation}{section}

\title{\huge\bfseries\textrm{Small noise may diversify collective motion in Vicsek model*}\footnotetext{*This work was supported by
the National Key Basic Research Program of China (973 program) under grant 2014CB845301/2/3, and by the National Natural Science Foundation of
China under grants No. 61203141 and 91427304.}\footnotetext{*Part
of this paper will been reported in The 34th Chinese Control Conference, 2015, Hangzhou, China.}}
\author{\normalsize\textrm {Ge CHEN, \emph{Member, IEEE}}\\
\footnotesize Key Laboratory of Systems and Control \&  National Center for Mathematics and Interdisciplinary Sciences,\\
\footnotesize Academy of Mathematics and Systems Science, Chinese Academy of Sciences, Beijing, 100190, China \\
\footnotesize Emails: chenge@amss.ac.cn}

\begin{document}
\maketitle
\emph{\textbf{Abstract}-}
Natural systems are inextricably affected by noise. Within recent decades, the manner in which noise affects the collective behavior of self-organized systems, specifically, has garnered considerable interest from researchers and developers in various fields. To describe the collective motion of multiple interacting particles, Vicsek \emph{et al.} proposed a well-known self-propelled particle (SPP) system, which exhibits a second-order phase transition from disordered to ordered motion in simulation; due to its non-equilibrium, randomness, and strong coupling nonlinear dynamics, however, there has been no rigorous analysis of such a system to date. To decouple systems consisting of deterministic laws and randomness, we propose a general method which transfers the analysis of these systems to the design of cooperative control algorithms. In this study, we rigorously analyzed the original Vicsek model under both open and periodic boundary conditions for the first time, and developed extensions to heterogeneous SPP systems (including leader-follower models) using the proposed method. Theoretical results show that SPP systems switch an infinite number of times between ordered and disordered states for any noise intensity and population density, which implies that the phase transition indeed takes a nontraditional form. We also investigated the robust consensus and connectivity of these systems. Moreover, the findings presented in this paper suggest that our method can be used to predict possible configurations during the evolution of complex systems, including turn, vortex, bifurcation and flock merger phenomena as they appear in SPP systems.


\section{Introduction}
``Natural systems are undeniably subject to random fluctuations, arising from either environmental
variability or thermal effects" \cite{sagu2007}. The manner in which noise affects the collective behavior of self-organized systems, which are shaped by the interplay of deterministic laws and randomness, has fascinated researchers in various fields such as catalysis, cosmology, biology, reactive mixing, colloidal chemistry, geophysics, electronic engineering, statistical physics, economics, and finance throughout the past several decades \cite{sagu2007,shin2001,Raser2005,Tsimring2014,Vicsek2012,Black1986}. The collective motion of groups of animals, for example, is a common (though highly remarkable) natural phenomenon that closely relates to this area of research. Schools of fish, flocks of birds, and groups of ants typically move in a highly orderly fashion that has been quantitatively described, for instance, by the now well-known self-propelled particle (SPP) system proposed by Vicsek \emph{et al.} \cite{Vicsek1995}.
In this system each agent moves with a constant speed, and at each time step
 adopts the average direction of motion of the other agents in their local neighborhood up to some added noise.
Using simulations,
Vicsek \emph{et al.} explored the relation between order, noise, and population density, and found that the SPP system exhibits a second-order phase transition from disordered to ordered motion concerning noise and population density under periodic boundary conditions \cite{Vicsek1995}.

The SPP system (also referred to as the \emph{Vicsek model}) is of interest to biologists, physicists, control theorists, and mathematicians because it captures common features of a number of real-world systems and is considered as a minimal model \cite{Vicsek2012}. For example, the SPP system's phase transition is similar to the ferromagnetic phase transition \cite{Vicsek1995,Toner2005} and to superconduction \cite{Alicea2005,Dmitriev2012}. Variations of the Vicsek model can be applied to study the collective motion of a wide range of biological systems such as cell colonies, flocks of birds, and swarms of locusts \cite{Buhl2006,Baskaran2009,Bialek2012,Belmonte2008}, and are also related to certain engineering applications such as the distributed computation and formation control of multi-agent systems [15-17]. Another important reason that the Vicsek model has become a common approach to theoretical research on complex systems is because it represents a simple, local rule of interaction that results in complex, global behavior.

To mathematically analyze the Vicsek model, its basic rule must be modified in existing research. Jadbabaie \emph{et al.} did so by omitting the noise and locally linearizing the updating equation of the movement direction of each agent \cite{Jad1}. This modification is adopted by other researchers \cite{sakvin,tang2007,Ren,Chen2014}. In a previous study, we introduced the percolation theory to investigate this modified system with a large population, and quantitatively described its smallest possible interaction radius (or population density under a scaling) for consensus  \cite{Chen2014}; this result shed light on the phase transition of the Vicsek model from disordered to ordered motion concerning population density in a short time period (provided the system encounters small noise), because when the population size is large, according to the law of large numbers, the effect of the noise is negligible in a brief time period. However, when the time period is lengthy, the cumulative effect of even relatively small amounts of noise cannot be omitted and do affect the system's global behavior. Until now, as mentioned above, the manner in which noise affects the collective behavior of the SPP system within longer time periods is unclear. Another modification is to assume that each agent can communicate with all of the others in the system at any time \cite{smale2007,Cucker2008,Carrillo2010,Shen2007,Park2010,Bernoff2013}. The literature also contains studies in which robust consensus is investigated by assuming the interaction between agents does not depend on the states of individual agents \cite{Wang2009,Shi2013,Tian2009,Munz2011,Khoo2009,CY2011}, as well as several reviews of  the Vicsek model and related concepts \cite{Vicsek2012,Yates2010,Hu2013}. To the best of our knowledge, however, no existing mathematical analysis of Vicsek-type models can maintain all three features of the original Vicsek model (refers to the SPP system in \cite{Vicsek1995}): self-driving, local interaction, and randomness.

Physicists mainly use hydrodynamics to analyze the Vicsek model. This method assumes that population size is infinite and approximates the Vicsek model to certain partial differential equations or stochastic partial differential equations \cite{Toner2005}. However, this approximation inevitably changes some inherent properties of the model, and can only represent some properties of the original system. Though the Vicsek model has been studied for twenty years and a number of works on the subject have provided valuable insight, physicists still lack a global understanding of it \cite{Solon2015}.

In this study, we attempted a global analysis of the original Vicsek model and a few heterogeneous SPP systems. The main contributions of this paper can be summarized as follows.

First, we propose a novel, general method of decoupling the heterogeneous self-organized system formulated by deterministic laws and randomness which widely exists in nature, engineering applications, societies, and economies \cite{sagu2007,shin2001,Raser2005,Tsimring2014,Vicsek2012,Black1986}. Our method transfers the analysis of such systems to the design of control algorithms, though the models do not contain any control input. Using the propose method, we rigorously analyze the original Vicsek model and create extensions to heterogeneous SPP systems (including the leader-follower model). We also provide a few clear answers to robust consensus and connectivity problems which are of interest in the field of
 multi-agent systems \cite{Jad1,Jad2007,Wang2009,Shi2013,Tian2009,Munz2011,Khoo2009,CY2011}.

In addition to analysis of final states, our method can also be used to predict possible configurations during the evolution of complex systems. As an example, we show that SPP systems can spontaneously generate turn, vortex, bifurcation, and flock merger phenomena. Our method is particularly adept at predicting events which happen with small probability in finite time that are difficult to observe through simulation, demonstrating potential application in complex engineering practices such as collision analysis/avoidance.

The results we present here also have significance in regard to physics and biology. We show that the Vicsek model switches an infinite number of times between ordered and disordered states for any noise intensity and population density, which indicates that even small noise may break the order of the system; this lends mathematical proof to the concept that randomness can result in non-equilibrium systems exhibiting anomalously large fluctuations \cite{Tsimring2014,Keizer1987}. The same result implies that the phase transition of the Vicsek model actually differs in form from the one traditionally assumed \cite{Vicsek1995,Czirak1999}. Furthermore, to some degree, our results provide an explanation for switches in group movement direction and large fluctuations of order parameters observed in locust swarm experiments at low and middle population densities \cite{Buhl2006}, and allows us to predict that these phenomena will continue to exist at high population densities when the time step is sufficiently large.

The rest of the paper is organized as follows: In Section \ref{problem}, we will introduce our model and give some definitions.
 Section \ref{transfer} provides a key method to analyze our models.
 The main results under open and periodic boundary conditions are put in Sections \ref{mrs_1} and \ref{mrs_2} respectively. In Section \ref{mrs_3} we give a theorem under an assumption.
Section \ref{sim} provides some simulations, and Section \ref{conclude} concludes this paper with future works.

\section{Models and definitions}\label{problem}
\subsection{The Original Vicsek Model}
The original Vicsek model
consists of $n$ autonomous agents moving in the plane with the same
speed $v(v>0)$,
where each agent $i$ contains two state variables: $X_i(t)=((x_{i1}(t),x_{i2}(t))\in\mathbb{R}^2$ and $\theta_i(t)\in
[-\pi,\pi)$, denoting its position and heading at time $t$ respectively.
Then the agent $i$'s velocity is $v(\cos\theta_i(t),\sin\theta_i(t))$ at time $t$.
Each agent's heading is updated according to a local rule
based on the average direction of its neighbors, and two agents are called neighbors if and only if their
distance is less than a pre-defined radius $r
(r>0)$. Let
$$\mathcal{N}_i(t):=\left\{j:\|X_i(t)-X_j(t)\|_2\leq r \right\}$$ denote the neighbor set
of agent $i$ at time $t$, where  $\|\cdot\|_2$ is the Euclidean norm. Following \cite{Vicsek1995}, the dynamics of the original Vicsek model can be formulated by
\begin{eqnarray}\label{m1_00}
\begin{aligned}
&\theta_{i}(t+1)={\rm{atan2}} \left(\sum_{j\in \mathcal{N}_{i}(t)}
\sin\theta_{j}(t),\sum_{j\in\mathcal{N}_{i}(t)}\cos
\theta_{j}(t)\right)+\zeta_i(t),
\end{aligned}
\end{eqnarray}
and
\begin{eqnarray}\label{m1}
\begin{aligned}
X_i(t+1)&=X_i(t)+V_i(t+1)=X_i(t)+v(\cos\theta_i(t+1),\sin\theta_i(t+1))
\end{aligned}
\end{eqnarray}
for all $i\in[1,n]$ and $t\geq 0$, where the function $atan2$ is the arctangent function with two arguments\footnote{Literature \cite{Vicsek1995} uses the $\arctan$ function here, but it should be not correct because the  quadrant information is lost.}, and $\{\zeta_i(t)\}$ is a random noise sequence  independently and uniformly distributed in a fixed interval whose midpoint is $0$. The system (\ref{m1_00})-(\ref{m1}) is called as the \emph{original Vicsek model}.
Let $X(t)=(X_1(t),X_2(t),\ldots,X_n(t))$ and $\theta(t)=(\theta_1(t),\theta_2(t),\ldots,\theta_n(t))$.
The original Vicsek model is very complex to analyze in mathematics. An important step forward in analyzing
this model was given by Jadbabaie \emph{et al.} in \cite{Jad1} who omitted
the noise item and locally linearized the updating rule (\ref{m1_00}) of the heading as follows:
 \begin{eqnarray*}\label{Jad_mod}
\theta_i(t+1)= \frac{1}{\mathcal{N}_{i}(t)}\sum_{j\in\mathcal{N}_{i}(t)} \theta_j(t).
\end{eqnarray*}

\subsection{Our Heterogeneous SPP systems}
To be more practical this paper will make some extensions to the original Vicsek model.
First we assume that each agent $i$ has different interaction radius $r_i>0$, and the interaction weight between two agents $i$ and $j$
is a non-negative function $f_{ij}(t)$ satisfying:\\
(i) $f_{ii}(t)> 0$ for all $i,t$, which means that each agent has a certain inertia; \\
(ii)  $f_{ij}(t)=0$ when $\|X_i(t)-X_j(t)\|_2>r_i$  for all $i,j,t$,
which indicates each agent cannot receive information directly from the ones out of its interaction radius.\\
Second we consider more general noise. Let $\xi_i(t)$ denote the new noise. Let
$\Omega^t=\Omega_n^t\subseteq \mathds{R}^{n\times(t+1)}$ be the sample space of $(\xi_i(t'))_{1\leq i \leq n, 0\leq t' \leq t}$,
and $\mathcal{F}^t=\mathcal{F}_n^t$ be its Borel $\sigma$-algebra. Additionally we define $\Omega^{-1}$ be the empty set.
Let $P=P_n$ be the probability measure on $\mathcal{F}^{\infty}$ for $(\xi_i(t'))_{1\leq i \leq n, t' \geq 0}$, so  the
 probability space is written as $( \Omega^{\infty},\mathcal{F}^{\infty},P)$. Throughout this paper we assume there exists a constant $\eta=\eta_n>0$ such that for all initial positions $X(0)$ and headings $\theta(0)$ and $t\geq 0$,  the joint probability density of $(\xi_1(t),\ldots,\xi_n(t))$
in the region $[-\eta,\eta]^n$ has a uniform lower bound $\underline{\rho}=\underline{\rho}(\eta,n)>0$ under all previous samples, i.e.,
for any real numbers $a_i, b_i$ with $-\eta\leq a_i < b_i \leq \eta$, $1\leq i \leq n$,
\begin{eqnarray}\label{noise_cond_2}
\begin{aligned}
&P\left(\bigcap_{i=1}^n \left\{ \xi_i(t)\in [a_i,b_i]\right\}| \forall w_{t-1} \in \Omega^{t-1} \right)\geq  \underline{\rho} \prod_{i=1}^n (b_i-a_i),~~~~\forall t\geq 0.
\end{aligned}
\end{eqnarray}
We would like to point out that in addition to the independent and uniform noise in the original Vicsek model, the new noise in (\ref{noise_cond_2}) also contains non-degenerate Gaussian white noise and some other bounded or unbounded noises.
With the above two extensions the equation (\ref{m1_00})  of the original Vicsek model is changed to
\begin{eqnarray}\label{m1_new}
&&\theta_{i}(t+1)={\rm{atan2}}\left(\sum_{j=1}^n f_{ij}(t)
\sin\theta_{j}(t),\sum_{j=1}^n f_{ij}(t)\cos
\theta_{j}(t)\right)+\xi_i(t).\nonumber
\end{eqnarray}
To simplify the exposition we record the system evolved by (\ref{m1}) and (\ref{m1_new}) as \emph{System I}.

We also consider the system whose updating equation of heading is
 \begin{eqnarray}\label{model1}
\theta_i(t+1)= \frac{1}{\sum_{j=1}^n f_{ij}(t)}\sum_{j=1}^n f_{ij}(t)\theta_j(t)+\xi_i(t).
\end{eqnarray}
For all $i\in[1,n]$ and $t\geq 0$, we restrict the value of heading $\theta_i(t)$  to the interval $[-\pi,\pi)$ by modulo $2\pi$ when it is out of this interval.
Similarly, we record the system evolved by (\ref{m1}) and (\ref{model1}) as \emph{System II}.
As a departure from existing modifications made for mathematical analysis, System II maintains the self-driving, local interaction, and randomness features of the original Vicsek model. We demonstrate below via simulations that System II in fact exhibits several properties similar to the original Vicsek model (Section \ref{sim}).

It is worth noting that Systems I and II can satisfy leader-follower relationships within a flock - for example, if agent $i$ is the follower of agent $j$, we can set $f_{ij}(t)$ to a large value and $f_{ik}(t)=0$ when $k\neq i,j$. The leader-follower relationship has been observed in real-world experiments \cite{Nagy2010}, and ours is the first study to investigate how noise affects order in this manner.
\subsection{Order, Robust Consensus and Connectivity}\label{Subsection_def2}
We first investigate how the noise affects the order.
Following \cite{Vicsek1995}, we define the order parameter $$\varphi(t):=\frac{1}{n}\big\|\sum_{i=1}^n\left(\cos\theta_i(t),\sin\theta_i(t)\right)\big\|_2$$
for all $t\geq 0$. Clearly, $\varphi(t)$ is close to its extreme value, $1$, indicating all the agents move in almost the same direction; when closer to $0$, $\varphi(t)$ indicates an absence of any collective alignment. Naturally, we say Systems I and II are \emph{ordered} at time $t$  when $\varphi(t)$ is close to $1$, and are \emph{disordered} when $\varphi(t)$ is close to $0$.

We also give an intuitive definition concerning the order:
\begin{de}
For any heading vector $\theta=(\theta_1,\theta_2,\ldots,\theta_n)\in [-\pi,\pi)^n$, define the length of the shortest interval which can cover it as
\begin{eqnarray*}
d_{\theta}:=\inf\left\{l\in[0,2\pi):  \mbox{there exists a constant } c\in[-\pi,\pi)\right.\\
  \left.\mbox{ such that } \theta_i \in[c,c+l] \mbox{ for all }1\leq i\leq n \right\},
  \end{eqnarray*}
  where $[c,c+l]:=[c,\pi)\cup [-\pi,c+l-2\pi]$ for the case of $c+l\geq\pi$.
\end{de}
This definition can also be understood as the \emph{maximum heading difference} in the flock. Clearly, $d_{\theta}$ is close to $0$ when all the agents move in almost the same direction.

The robust consensus has attracted much attention in multi-agent system research \cite{Wang2009,Shi2013,Tian2009,Munz2011,Khoo2009,CY2011}. Wang and Liu \cite{Wang2009}, for example, provided a definition of robust consensus for systems whose network topologies do not couple the agents' states. We adapted this definition to suit our model as follows:
\begin{de}\label{def_robust_consensus}
System I (or II) achieves robust consensus if there exists a function $g(\cdot)$ satisfying $\lim_{x\rightarrow 0^+}g(x)=0$, such that for any $\eta>0$ and $\omega\in \Omega^{\infty}$, $$\limsup_{t\rightarrow\infty} d_{\theta(t)} \leq g(\eta).$$
\end{de}
This paper will study whether the robust consensus can be reached.

The connectivity of the network topology is a key issue for consensus of multi-agent systems. For System I (or II), let
$G(t)=G(\mathcal{X},\mathcal{E}(t))$ denote its underlying graph at time $t$, where the vertex set $\mathcal{X}$ is the $n$ agents, and the edge set $\mathcal{E}(t)=\{(j,i):\|X_i(t)-X_j(t)\|_2\leq r_i\}.$ Note that $G(t)$ is a directed graph in our heterogeneous system.
A directed graph is said to be \emph{strongly connected} if there exists at least one path in each direction between each pair of vertices of the graph.
Let $\widetilde{G}(t)$
denote the graph obtained by replacing all directed edges of $G(t)$ with undirected edges.
Clearly, the graph $\widetilde{G}(t)$ is undirected. An undirected graph is said to be \emph{connected} if there exists at least one path between its any two vertices. If  $\widetilde{G}(t)$ is connected, then $G(t)$ is said to be \emph{weakly connected}. If a directed graph is strongly connected, of course it is also weakly connected.

 Given two graphs $G(\mathcal{X},\mathcal{E}_1)$ and $G(\mathcal{X},\mathcal{E}_2)$, define $G(\mathcal{X},\mathcal{E}_1)\cup G(\mathcal{X},\mathcal{E}_2):=G(\mathcal{X},\mathcal{E}_1\cup \mathcal{E}_2)$.
  Following \cite{Shi2013}, we give the definition of uniformly joint weak connectivity as follows:
\begin{de}
The graph sequence $\{G(t)\}_{t=0}^{\infty}$ is said to be uniformly jointly weakly connected if there exists an integer $T>0$ such that
$\cup_{k=t}^{t+T}\widetilde{G}(k)$ is connected for any $t>0$.
\end{de}

The assumption of uniformly joint connectivity is widely considered a sufficient condition of consensus in multi-agent systems \cite{Jad1,sakvin,Ren,Wang2009,Shi2013,Tian2009,Munz2011,CY2011,Huang2012}. For systems whose topologies are coupled with states, whether this assumption can be satisfied remains a quite interesting problem. In this paper, we will show with probability $1$ the underlying graphs of our heterogeneous SPP systems are not uniformly jointly weakly connected and, of course, they are not uniformly jointly connected in a homogeneous case.

\subsection{Turn, Vortex, Bifurcation and Merging}

Turning, bifurcation, and merging of flocks are very common phenomena in nature. These phenomena have been studied under the well-known Boid model using simulations \cite{Reynold1987}. Olfati-Saber \cite{Saber2006}, for example, provided a specific flocking algorithm that can produce bifurcation and merger behavior by adding a global leader and a few obstacles.We will show that the SPP model can spontaneously produce these phenomena, which are difficult to precisely define.\\
\emph{Turn} and \emph{Vortex}: All agents of a flock gradually change their headings from one angle to another in a finite amount of time, where the difference of the two angles is larger than a certain value (for example, $\pi/2$). During this time, all the agents stay nearly synchronized, i.e., their headings are almost the same at each time step. A turn with change in angle exceeding $2\pi$ is called a \emph{vortex}.\\
\emph{Bifurcation}: A group of agents whose directions of motion are almost same may separate into two groups with different directions, where the agents in each group are nearly synchronized.\\
\emph{Merging}: Two groups of agents moving in different directions may merge into one group that moves in almost the same direction.

\section{Transform to Robust Cooperative Control}\label{transfer}
To analyze Systems I and II, we first must construct two robust control systems capable of transforming the analysis to the design of control algorithms. For $i=1,\ldots,n$ and $t\geq 0$, let $\delta_i(t)\in(0,\eta)$ be an arbitrarily given real number,
$u_i(t)\in [-\eta+\delta_i(t),\eta-\delta_i(t)]$ denote a bounded control input, and $b_i(t)\in [-\delta_i(t),\delta_i(t)]$ denote the
parameter uncertainty.
For System I we construct the following control system
\begin{eqnarray}\label{model2_new}
\left\{%
\begin{array}{ll}
    \theta_i(t+1)={\rm{atan2}}(\sum_{j=1}^n f_{ij}(t)
\sin\theta_{j}(t),\sum_{j=1}^n f_{ij}(t)\cos
\theta_{j}(t))+u_i(t)+b_i(t), \\
    X_{i}(t+1)=X_{i}(t)+v(\cos\theta_{i}(t+1),\sin\theta_i(t+1)),
\end{array}%
\right.
\end{eqnarray}
and for System II we do the same:
\begin{eqnarray}\label{model2}
\left\{%
\begin{array}{ll}
    \theta_i(t+1)=\frac{1}{\sum_{j=1}^n f_{ij}(t)}\sum_{j=1}^n f_{ij}(t)\theta_j(t)+u_i(t)+b_i(t), \\
    X_{i}(t+1)=X_{i}(t)+v(\cos\theta_{i}(t+1),\sin\theta_i(t+1)).
\end{array}%
\right.
\end{eqnarray}



Let $ S^*:= \mathds{R}^{2n}\times [-\pi,\pi)^n$(or $[0,L]^{2n}\times [-\pi,\pi)^n$ for the periodic boundary case defined in Section \ref{mrs_2}) be the state space of  $(X(t), \theta(t))$ for all $t\geq 0$. Given $ S_1\subseteq S^*$, we say $ S_1$ \emph{is reached at time $t$} if  $(X(t),\theta(t))\in  S_1$, and i\emph{s reached in the time $[t_1,t_2]$} if there exists $t'\in [t_1,t_2]$ such that  $ S_1$ is reached at time $t'$.

\begin{de}\label{def_reach}
Let  $ S_1, S_2\subseteq  S^*$ be two state sets.
Under protocol (\ref{model2_new}) (or (\ref{model2})), $S_1$ is said to be finite-time robustly reachable from $ S_2$ if: For any $(\theta(0),X(0))\in S_2$, $S_1$ is reached at time $0$, or there exist constants $T>0$ and $\varepsilon\in(0,\eta)$ such that we can find $\delta_i(t)\in[\varepsilon,\eta)$ and $u_i(t)\in [-\eta+\delta_i(t),\eta-\delta_i(t)]$, $1\leq i\leq n$, $0\leq t<T$ which guarantees that
$ S_1$ is reached in the time $[1,T]$ for arbitrary $b_i(t)\in [-\delta_i(t),\delta_i(t)]$,
$1\leq i\leq n$, $0\leq t<T$.
\end{de}

\begin{rem}\label{rem_def_reach}
Under normal circumstances, $\delta_i(t)$ serves as a constant $\varepsilon>0$ for all $1\leq i\leq n$ and $0\leq t<T$ to guarantee that system (\ref{model2_new}) (or (\ref{model2})) robustly reaches a designated state set in time $[1,T]$. Additionally, $\varepsilon$  can be set to a sufficiently small value such that the uncertainty item $b_i(t)$ does not affect the system's macro states, such as the ordered or disordered states in finite time.
\end{rem}

The following lemma establishes a connection between System I and protocol (\ref{model2_new}) , and also between System II and protocol (\ref{model2}).

\begin{lem}\label{robust}
Let $ S_1,\ldots,  S_k\subseteq  S^*$, $k\geq 1$ be state sets and assume they are finite-time robustly reachable from $ S^*$
under protocol (\ref{model2_new}) (or (\ref{model2})). Suppose the initial positions $X(0)$ and headings $\theta(0)$ are arbitrarily given. Then for System I (or II):\\
(i)~ With probability $1$  there exists an infinite sequence $t_1<t_2<\ldots$ such that $ S_j$ is reached at time
$t_{lk+j}$ for all $j=1,\ldots,k$ and $l\geq 0$.\\
(ii)~ There exist constants $T>0$ and $c\in (0,1)$ such that
$$P\left(\tau_i-\tau_{i-1}>t\right)\leq c^{\lfloor t/T\rfloor}, \forall i,t\geq 1,$$
where $\tau_0=0$ and
$\tau_i:=\min\{t: \mbox{ there exist }\tau_{i-1}<t_1'<t_2'<\cdots<t_k'=t \mbox{ such that for all }j\in[1,k], S_j \mbox{ is reached at time }t_j'\}$ for $i\geq 1$.
\end{lem}
\begin{proof}
(i) Throughout this proof we assume the initial state is arbitrarily given. We recall that
$\Omega^t\subseteq \mathds{R}^{n\times(t+1)}$ is the sample space of $(\xi_i(t'))_{1\leq i \leq n, 0\leq t' \leq t}$.
 Under System I (or II),
the values of $X(t)$ and $\theta(t)$  are determined by the sample $w_{t-1}\in\Omega^{t-1}$, so
for any $t\geq 1$ and $j\in[1,n]$
we can set $$\Omega_j^{t-1}:=\left\{w_{t-1}\in\Omega^{t-1}: (X(t),\theta(t))(w_{t-1})\in S_j\right\}$$
to be the subset of $\Omega^{t-1}$ such that $S_j$ is reached at time $t$. Thus,
\begin{eqnarray}\label{rob_0}
\begin{aligned}
P\left(\left\{\mbox{$ S_j$ is reached at time $t$}\right\}|\forall w_{t-1}'\in \Omega_j^{t-1}\right)=1.
\end{aligned}
\end{eqnarray}
Also, by our assumption $ S_j$ is finite-time robustly reachable under protocol (\ref{model2_new}) (or (\ref{model2})), so with Definition \ref{def_reach} there exist constants $T_j\geq 2$ and $\varepsilon_j\in(0,\eta)$ such that for any $t\geq 0$ and $(X(t),\theta(t))\notin S_j$,
 we can find parameters $\delta_i(t')\in [\varepsilon_j,\eta)$ and control inputs $u_i(t')\in [-\eta+\delta_i(t'),\eta-\delta_i(t')]$, $1\leq i\leq n$, $t\leq t'\leq t+T_j-2$ with which the set $ S_j$ is reached in the time $[t+1,t+T_j-1]$ for any uncertainties $b_i(t')\in [-\delta_i(t'),\delta_i(t')]$, $1\leq i\leq n, t\leq t'\leq t+T_j-2$.
 This acts on System I (or II) indicating that for any $w_{t-1}^*\in (\Omega_j^{t-1})^c$,
 \begin{eqnarray}\label{rob_1}
\begin{aligned}
&P\left(\left\{\mbox{$ S_j$ is reached in $[t+1, t+T_j-1]$}\right\}| w_{t-1}^*\right)\\
&\geq P\left(\bigcap_{t\leq t'\leq t+T_j-2}\bigcap_{1\leq i\leq n}\left\{\xi_i(t')\in [u_i(t')-\delta_i(t'),u_i(t')+\delta_i(t')] \right\}|w_{t-1}^*\right).
\end{aligned}
\end{eqnarray}
Here $(\Omega_j^{t-1})^c$ means the complement set of $\Omega_j^{t-1}$.
Define
$$F_t:=\bigcap_{1\leq i\leq n}\left\{\xi_i(t)\in [u_i(t)-\delta_i(t),u_i(t)+\delta_i(t)] \right\}.$$
By the Bayes' theorem we can get
\begin{eqnarray}\label{rob_1_2}
\begin{aligned}
\mbox{the right side of (\ref{rob_1})}&=P\left( F_t|w_{t-1}\right)\prod_{t'=t+1}^{t+T_j-2} P\left(F_{t'}|\bigcap_{t\leq l<t'}F_{l},w_{t-1}^*\right)\\
&\geq \prod_{t'=t}^{t+T_j-2}\left[\underline{\rho} \prod_{i=1}^n (2\delta_i(t'))\right]\geq \underline{\rho}^{T_j-1} \left(2 \varepsilon_j \right)^{n (T_j-1)},
\end{aligned}
\end{eqnarray}
where the first inequality uses (\ref{noise_cond_2}) and the fact of $-\eta\leq u_i(t')-\delta_i(t')<u_i(t')+\delta_i(t')\leq \eta$ for
$1\leq i\leq n$ and $t\leq t'<t+T_j$.
Define the event
$$E_{j,t}:=\left\{\mbox{$ S_j$ is reached in $[t, t+T_j-1]$}\right\},$$
Combining (\ref{rob_0}), (\ref{rob_1}) and (\ref{rob_1_2}) yields
\begin{eqnarray}\label{rob_1_3}
\begin{aligned}
&P\left(E_{j,t} | \forall w_{t-1}\in \Omega^{t-1}\right)\geq \underline{\rho}^{T_j-1} \left(2 \varepsilon_j \right)^{n (T_j-1)},
\end{aligned}
\end{eqnarray}
where the second inequality uses (\ref{noise_cond_2}) and the fact of $-\eta\leq u_i(t')-\delta_i(t')<u_i(t')+\delta_i(t')\leq \eta$ for
$1\leq i\leq n$ and $t\leq t'\leq t+T_j-2$.
Using the Bayes' theorem and (\ref{rob_1_3}) we have for any $w_{t-1}\in \Omega^{t-1}$,
\begin{eqnarray}\label{rob_2}
\begin{aligned}
P\left(\bigcap_{j=1}^k E_{j,t+\sum_{l=1}^{j-1} T_l} | w_{t-1}\right)&=P\left( E_{1,t}|w_{t-1}\right)\prod_{j=2}^{k} P\left(E_{j,t+\sum_{l=1}^{j-1} T_l}\Big|\bigcap_{l=1}^{j-1} E_{l,t+\sum_{p=1}^{l-1} T_p},w_{t-1}\right)\\
&\geq \prod_{j=1}^{k} \left[\underline{\rho}^{T_j-1} \left(2 \varepsilon_j \right)^{n (T_j-1)}\right]:=c.
\end{aligned}
\end{eqnarray}

Set $E_t:=\bigcap_{j=1}^k E_{j,t+\sum_{l=1}^{j-1} T_l}$ and $T:=T_1+T_2+\ldots+T_k$.
For any integer $M>0$,
using Bayes' theorem again and  (\ref{rob_2}) we have
\begin{eqnarray}\label{rob_3}
\begin{aligned}
P\Big(\bigcap_{m=M}^{\infty}E_{mT}^c\Big)&=P\left(E_{MT}^c\right)\prod_{m=M+1}^{\infty}P\Big(E_{mT}^c\big| \bigcap_{M\leq m'<m}E_{m'T}^c\Big)\\
&=\left[1-P\left(E_{MT}\right)\right]\prod_{m=M+1}^{\infty}\Big[1-P\Big(E_{mT}\big| \bigcap_{M\leq m'<m}E_{m'T}^c\Big)\Big]\\
&\leq \prod_{m=M}^{\infty}\left(1-c\right)=0,
\end{aligned}
\end{eqnarray}
which indicates that with probability $1$ there exits an infinite sequence $m_1<m_2<\ldots$ such that $E_{m_l T}$ occurs for all $l\geq 1$.
Here $E_{mT}^c$ is the complement set of $E_{mT}$.
By the definition of $E_t$, for each $l\geq 0$ we can find a time sequence $t_{lk+j}\in [m_l T+\sum_{p=1}^{j-1} T_p, m_l T+\sum_{p=1}^{j} T_p-1]$, $1\leq j\leq k$ such that $S_j$ is reached at time $t_{lk+j}$.

(ii)
 For any $M\geq 0$ and $i>0$, the event $\tau_i-\tau_{i-1}>MT$ means that $E_t$ does not happen for all $t\in[\tau_{i-1}+1,\tau_{i-1}+1+(M-1)T]$.
 By the total probability theorem and (\ref{rob_3})  we have
\begin{eqnarray*}\label{rob_4}
\begin{aligned}
P\left\{\tau_i-\tau_{i-1}>MT \right\}&\leq P\left(\bigcap_{m=0}^{M-1}E_{\tau_{i-1}+1+mT}^c\right)\\
&=\sum_{t=0}^{\infty}P(\tau_{i-1}=t) P\left(\bigcap_{m=0}^{M-1}E_{t+1+mT}^c\right) \\
&\leq \left(1-c\right)^M\sum_{t=0}^{\infty}P(\tau_{i-1}=t)=\left(1-c\right)^M,
\end{aligned}
\end{eqnarray*}
so
\begin{eqnarray*}\label{rob_5}
\begin{aligned}
P\left(\tau_i-\tau_{i-1}>t \right)&\leq P\left(\tau_i-\tau_{i-1}>\lfloor \frac{t}{T} \rfloor T \right)\leq \left(1-c\right)^{\lfloor t/T\rfloor}.
\end{aligned}
\end{eqnarray*}
\end{proof}

Specially, for the case of $k=1$, from Lemma \ref{robust} we can get the following corollary:
\begin{cor}\label{robust2}
Let $ S \subseteq  S^*$ be a state set and assume it is finite-time robustly reachable from $S^c$
under protocol  (\ref{model2_new}) (or (\ref{model2})).
 Suppose the initial positions $X(0)$ and headings $\theta(0)$ are arbitrarily given.
Then for System I (or II):\\
(i)~With probability $1$ $ S$ will be reached an infinite number of times.\\
(ii)~There exist constants $T>0$ and $c\in (0,1)$ such that
$$P\left(\tau_i-\tau_{i-1}>t \right)\leq c^{\lfloor t/T\rfloor}, \forall i,t\geq 1,$$
where $\tau_0:=0$ and
$\tau_i:=\min\{t>\tau_{i-1}: S \mbox{ is reached at time }t \}$ for $i\geq 1$.
\end{cor}
\begin{proof}
Because $S$ is finite-time robustly reachable from $S^c$, of course it is also finite-time robustly reachable from $S^*$. From Lemma \ref{robust}
our result can be deduced directly.
\end{proof}

\begin{rem}
 Using Lemma \ref{robust} and Corollary \ref{robust2}, we can transfer the analysis of Systems I and II to design the robust control algorithms
for protocols  (\ref{model2_new}) and (\ref{model2}). In fact, based on Remark \ref{rem_def_reach} we can choose suitable parameters such that the
uncertainty item $b_i(t)$ does not affect the system's macro states in finite time, so the analysis of Systems I and II can be transformed to the design of the controls of the protocols  (\ref{model2_new}) and (\ref{model2}).
\end{rem}

\begin{rem}
The methods in  Lemma \ref{robust} and Corollary \ref{robust2} do not depend on detailed expressions of the systems. In fact, for the system
$x(t+1)=f(x(t),\xi(t))\in \mathds{R}^n$ with the noise $\xi(t)\in\mathds{R}^m$, we can apply the proposed methods to simplify the analysis - specifically, to predict possible configurations during the system's evolution and its final states.
\end{rem}

\section{Analysis under Open Boundary Conditions}\label{mrs_1}

This section will give some results under open boundary conditions of positions of agents, which indicates that all the agents can move on $\mathbb{R}^2$ without boundary limitation.
Throughout this section we make the following assumption:

\textbf{(A1)} Assume the population size $n\geq 2$, the parameters $\eta>0$, $\underline{\rho}>0$, $v>0$, $r_i\geq 0$, $1\leq i\leq n$, and the initial positions $X(0)\in\mathbb{R}^{2n}$ and headings $\theta(0)\in [-\pi,\pi)^n$ are arbitrarily given.


We also need introduce some definitions.
For any $t\geq 0$ and $1\leq i\leq n$, set
\begin{eqnarray*}
\widetilde{\theta}_i(t)=\left\{%
\begin{array}{ll}
{\rm{atan2}}(\sum_{j=1}^n f_{ij}(t)
\sin\theta_{j}(t),\sum_{j=1}^n f_{ij}(t)\\
\cos\theta_{j}(t))~\rm{for~System~I~and~protocol~(\ref{model2_new})},\\
\frac{1}{\sum_{j=1}^n f_{ij}(t)}\sum_{j=1}^n f_{ij}(t)\theta_j(t) \\
~~~~~~\rm{for~System~II~and~protocol~(\ref{model2})}.
\end{array}%
\right.
\end{eqnarray*}
Let $X=(X_1,\ldots,X_n)\in\mathbb{R}^{2n}$ and $\theta=(\theta_1,\ldots,\theta_n)\in [-\pi,\pi)^n$. For any $\alpha>0$, define
\begin{eqnarray*}
&& S_{\alpha}^1:=\big\{(X,\theta)\in S^*:\max_{1\leq i\leq n}|\theta_i|\leq \frac{\alpha}{2}\big\}.
\end{eqnarray*}
We see $S_{\alpha}^1$ is a set of ordered states when $\alpha$ is small.
The following Lemmas \ref{lem1} and \ref{lem1_2}  describe a transition to the ordered state for protocols (\ref{model2}) and (\ref{model2_new}) respectively.

\begin{lem}\label{lem1}
Assume that (A1) holds. Then
for any $\alpha>0$, $S_{\alpha}^1$ is finite-time robustly reachable from $ S^*$ under protocol (\ref{model2}).
\end{lem}
\begin{proof}
Without loss of generality we assume $\alpha\in(0,\eta]$. The main idea of this proof is: For each agent $i$, if its neighbors' average heading $\widetilde{\theta}_i(t)$ is larger than an upper bound, we set $u_i(t)$ be a negative input; if $\widetilde{\theta}_i(t)$ is less than a lower bound, we set $u_i(t)$ be a positive input; otherwise we select a control input such that $\theta_i(t+1)$ will be in the interval $[-\alpha/2,\alpha/2]$.
With this idea, for $t\geq 0$ and $1\leq i \leq n$  we choose
\begin{eqnarray}\label{lem1_1}
&&\left(\delta_i(t),u_i(t)\right)=\left\{%
\begin{array}{ll}
(\eta/4,-3\eta/4)~~~~~\mbox{if}~\widetilde{\theta}_i(t)> \eta-\alpha/2,\nonumber\\
(\alpha/2,-\widetilde{\theta}_i(t))~~~~\mbox{if}~\widetilde{\theta}_i(t)\in [\alpha/2-\eta,\eta-\alpha/2],\\
(\eta/4,3\eta/4)~~~~~~~~\mbox{if}~\widetilde{\theta}_i(t)< \alpha/2-\eta.
\end{array}%
\right.
\end{eqnarray}
Then it can be computed that
\begin{eqnarray}\label{lem1_2}
u_i(t)\in [-\eta+\delta_i(t),\eta-\delta_i(t)], ~~\forall 1\leq i\leq n, t\geq 0,
\end{eqnarray}
which means our choice of $(u_i(t),\delta_i(t))$ meets their requirements in Definition \ref{def_reach}. Define
$$\theta_{\max}(t):=\max_{1\leq i \leq n}\theta_i(t)~~~~\mbox{and}~~~~\theta_{\min}(t):=\min_{1\leq i \leq n}\theta_i(t).$$
If $\theta_{\max}(t)>\alpha/2+\eta/2$ we can get
\begin{eqnarray}\label{lem1_3}
\theta_{\max}(t+1)\leq \theta_{\max}(t)-\frac{\eta}{2}.
\end{eqnarray}
That is because if there exists $i\in [1,n]$ such that
\begin{eqnarray}\label{lem1_4}
\theta_i(t+1)>\theta_{\max}(t)-\frac{\eta}{2}>\frac{\alpha}{2},
\end{eqnarray}
by $\theta_i(t+1)=\widetilde{\theta}_i(t)+u_i(t)+b_i(t)$ and (\ref{lem1_1}) we have $\widetilde{\theta}_i(t)>\eta-\alpha/2$ and $u_i(t)+b_i(t)\in [-\eta,-\eta/2]$. But at the same time, by the definition of $\widetilde{\theta}_i(t)$ we have $\widetilde{\theta}_i(t)\leq \theta_{\max}(t)$, so
$$\theta_i(t+1)\leq\theta_{\max}(t)+u_i(t)+b_i(t)\leq \theta_{\max}(t)-\frac{\eta}{2},$$
which is contradictory with the first inequality of (\ref{lem1_4}).

Similar to (\ref{lem1_3}), we can get that if $\theta_{\min}(t)<-\alpha/2-\eta/2$ then
\begin{eqnarray}\label{lem1_5}
\theta_{\min}(t+1)\geq \theta_{\min}(t)+\frac{\eta}{2}.
\end{eqnarray}
Combining this with (\ref{lem1_3}) we have if  $\max_{1\leq i \leq n}|\theta_i(t)|>\alpha/2+\eta/2$ then
\begin{eqnarray}\label{lem1_6}
\max_{1\leq i \leq n}|\theta_i(t+1)|\leq \max_{1\leq i \leq n}|\theta_i(t)|-\frac{\eta}{2}.
\end{eqnarray}
Also, if $\max_{1\leq i \leq n}|\theta_i(t)|\leq \alpha/2+\eta/2$, by (\ref{lem1_1})
\begin{eqnarray}\label{lem1_7}
\max_{1\leq i \leq n}|\theta_i(t+1)|\leq \alpha/2.
\end{eqnarray}
Let $t_1:=\lceil \frac{2\pi-\alpha}{\eta}\rceil$.  By (\ref{lem1_6}), (\ref{lem1_7}) and with the fact of $\max_{1\leq i \leq n}|\theta_i(t)|\leq \pi$, we can get
\begin{eqnarray}\label{lem1_8}
\max_{1\leq i \leq n}|\theta_i(t_1)|\leq \alpha/2.
\end{eqnarray}
Combining (\ref{lem1_8}), (\ref{lem1_1}) and (\ref{lem1_2}) we have $ S_{\alpha}^1$ is robustly reached at time $t_1$ from any initial state under protocol (\ref{model2}).
\end{proof}

\begin{lem}\label{lem1_2}
Suppose that (A1) holds and there exists a constant $\eta>\frac{\pi}{2}-\frac{\pi}{n}$ satisfying (\ref{noise_cond_2}). Then for any $\alpha>0$,  $S_{\alpha}^1$ is finite-time robustly reachable from $ S^*$
under protocol (\ref{model2_new}).
\end{lem}
\begin{proof}
Compared to the proof of Lemma \ref{lem1},  the biggest difference of this proof is to control the maximum heading difference less than $\pi$ at the beginning time.
In this proof the angle $a\in[b,c]$ means $a\mod 2\pi$ belongs to the set of the elements in $[b,c]$ module $2\pi$.  Compute $\widetilde{\theta}_i(0)$, $1\leq i \leq n$, and the length of the shortest interval which can cover them must be not bigger than $2\pi(1-\frac{1}{n})$. Let $\theta^*$ be the middle point of this interval, then $\widetilde{\theta}_i(0)$, $1\leq i \leq n$, are all in $[\theta^*-\pi(1-\frac{1}{n}),\theta^*+\pi(1-\frac{1}{n})]$. Set $\varepsilon_1:=\min\{\frac{1}{3}(\eta-\frac{\pi}{2}+\frac{\pi}{n}),\frac{\pi}{8}  \}$.
For $1\leq i \leq n$, we choose $\delta_i(0)=\varepsilon_1$ and
\begin{eqnarray*}
&&u_i(0)=\left\{%
\begin{array}{ll}
-2\varepsilon_1-\frac{n-2}{2(n-1)}(\widetilde{\theta}_i(0)-\theta^*)\\
~~~~~\mbox{if}~\widetilde{\theta}_i(0)-\theta^*\in [0,\pi(1-\frac{1}{n})],\nonumber\\
2\varepsilon_1-\frac{n-2}{2(n-1)}(\widetilde{\theta}_i(0)-\theta^*)\\
~~~~~\mbox{if}~\widetilde{\theta}_i(0)-\theta^*\in [-\pi(1-\frac{1}{n}),0).
\end{array}%
\right.
\end{eqnarray*}
From this we can compute for $1\leq i \leq n$,
\begin{eqnarray*}
\begin{aligned}
u_i(0)&\geq -2\varepsilon_1-\frac{n-2}{2(n-1)}\pi(1-\frac{1}{n})=-2\varepsilon_1-\frac{\pi}{2}+\frac{\pi}{n}\geq -\eta+\varepsilon_1,
\end{aligned}
\end{eqnarray*}
and similarly $u_i(0) \leq \eta-\varepsilon_1,$
which indicates the condition of  $u_i(0)\in [-\eta+\delta_i(0),\eta-\delta_i(0)]$ in Definition \ref{def_reach} is satisfied.
Also, we can get
\begin{eqnarray*}
\begin{aligned}
\widetilde{\theta}_i(0)+u_i(0)-\theta^*&\in \left[\min\{-2\varepsilon_1, 2\varepsilon_1-\frac{\pi}{2}\},\max\{2\varepsilon_1,-2\varepsilon_1+\frac{\pi}{2}\}\right]\\
&=\left[2\varepsilon_1-\frac{\pi}{2},-2\varepsilon_1+\frac{\pi}{2}\right]
\end{aligned}
\end{eqnarray*}
and so with (\ref{model2_new})
\begin{eqnarray*}
\begin{aligned}
\theta_i(1)-\theta^*\in [\varepsilon_1-\frac{\pi}{2},-\varepsilon_1+\frac{\pi}{2}],~~~~\forall 1\leq i\leq n,
\end{aligned}
\end{eqnarray*}
which indicates that
\begin{eqnarray*}
\widetilde{\theta}_i(1)-\theta^*\in [\varepsilon_1-\frac{\pi}{2},-\varepsilon_1+\frac{\pi}{2}],~~~~\forall 1\leq i\leq n.
\end{eqnarray*}

Next we control all the headings of the agents to the neighborhood of $\theta^*$.
Let $\varepsilon_2:=\min\{\frac{\pi}{8},\frac{\eta}{4},\frac{\alpha}{2}\}$ and set $t_1:=\lceil \frac{\frac{\pi}{2}-\eta+\varepsilon_2}{\eta-2\varepsilon_2}\rceil+2$.
For $1\leq t<t_1$ and $1\leq i \leq n$  we choose $\delta_i(t)=\varepsilon_2$ and
\begin{eqnarray}
u_i(t)=\left\{%
\begin{array}{ll}
-\eta+\varepsilon_2~~~\mbox{if}~\widetilde{\theta}_i(t)-\theta^*\in (\eta-\varepsilon_2,\frac{\pi}{2})\nonumber\\
\theta^*-\widetilde{\theta}_i(t)~\mbox{if}~\widetilde{\theta}_i(t)-\theta^*\in [\varepsilon_2-\eta,\eta-\varepsilon_2],\\
\eta-\varepsilon_2~~~~~~\mbox{if}~\widetilde{\theta}_i(t)-\theta^*\in (\frac{-\pi}{2},\varepsilon_2-\eta).
\end{array}%
\right.
\end{eqnarray}
With almost the  same process of (\ref{lem1_2})-(\ref{lem1_8})  we have
\begin{eqnarray*}
\theta_i(t_1)-\theta^*\in [-\varepsilon_2,\varepsilon_2],~~~~\forall 1\leq i\leq n.
\end{eqnarray*}

Finally we control all the headings of the agents to the neighborhood of $0$. Without loss of generality we assume $\theta^*\in [-\pi,0]$. For $t\geq t_1$ and $1\leq i \leq n$  we choose $\delta_i(t)=\varepsilon_2$ and
\begin{eqnarray}
u_i(t)=\left\{%
\begin{array}{ll}
\eta-\varepsilon_2~~~~~\mbox{if}~\widetilde{\theta}_i(t)\in [-\pi-\varepsilon_2,\varepsilon_2-\eta)\nonumber\\
-\widetilde{\theta}_i(t)~~~~\mbox{otherwise},
\end{array}%
\right.
\end{eqnarray}
and can get that $\theta_i(t_2)\in [-\varepsilon_2,\varepsilon_2]$, $1\leq i \leq n$ with $t_2:=\lceil\frac{\pi}{\eta-2\varepsilon_2}\rceil+t_1$.
\end{proof}

\begin{rem}
The condition of  $\eta>\frac{\pi}{2}-\frac{\pi}{n}$ in Lemma \ref{lem1_2} can be satisfied for any non-degenerate Gaussian white noise sequence. Further, we posit that Lemma \ref{lem1_2} also holds for any $\eta>0$. Strict proof of this does not come easily, however. In essence,
protocol (\ref{model2_new}) without the control input and parameter uncertainty is an isotropic system whose transition
from disordered to ordered state, called ``symmetry breaking" in physics,  is rather difficult to analyze compared to the anisotropic protocol (\ref{model2})-
so we add a lower bound to the span of the control input of protocol (\ref{model2_new})  to break the system's
symmetry.
\end{rem}

The following lemma describes a connection between the order parameter and  the maximum heading difference.
\begin{lem}\label{lem2}
For any $\varepsilon\in(0,1)$ and $\theta(t)\in[-\pi,\pi)^n$, if $d_{\theta(t)}\leq \arccos(1-\varepsilon)^2$ then the order function
$\varphi(t)\geq 1-\varepsilon$.
\end{lem}
\begin{proof}
By the definition of $\varphi(t)$ we have
\begin{eqnarray*}
\begin{aligned}
\varphi(t)&=\frac{1}{n}\big\|\left(\sum_{i=1}^n\cos\theta_i(t),\sum_{i=1}^n \sin\theta_i(t)  \right)\big\|_2\\
&=\frac{1}{n} \sqrt{\sum_{i,j}\cos\left[\theta_i(t)-\theta_j(t)  \right] }\\
&\geq \sqrt{\cos \left( \arccos(1-\varepsilon)^2 \right)}\geq 1-\varepsilon.
\end{aligned}
\end{eqnarray*}
\end{proof}

For any $\varepsilon>0$, define
\begin{eqnarray*}
S_{\varepsilon}^2:=\big\{(X,\theta)\in S^*: \frac{1}{n}\big\|\sum_{i=1}^n(\cos\theta_i,\sin\theta_i\big)\big\|_2\leq \varepsilon\big\}.
\end{eqnarray*}
Then $S_{\varepsilon}^2$ is a set of disordered states providing $\varepsilon$ close to $0$.
The following lemma describes a transition from ordered states to disordered states.
\begin{lem}\label{lem3}
Assume (A1) holds. Then
for any $\varepsilon>0$,  $ S_{\varepsilon}^2$ is finite-time robustly reachable from $ S_{\eta}^1$ under both protocols (\ref{model2_new}) and (\ref{model2}).
\end{lem}
To outline the proof of Lemma  \ref{lem3}, we first divided the agents into different sets, then we controlled the agents' headings in different sets to have a certain amount of disparity, breaking all the communications between different sets after a finite time. Next, we controlled the headings in each set to a designed angle so that the order parameter of the system became very small. For detailed proof, see Appendix \ref{App_lem3}.

We assert through the following theorem that the order parameter will switch an infinite number of times between very large and very small. Please note that the large order parameter indicates ordered states, and the small order parameter indicates disordered states.

\begin{thm}\label{result_1}
Assume (A1) holds and let $\varepsilon\in(0,1)$ be a constant arbitrarily given. Then for System II (or System I with $\eta>\frac{\pi}{2}-\frac{\pi}{n}$), with probability $1$ there exists an infinite time sequence $t_1<t_2<\cdots$  such that
 \begin{eqnarray*}
\varphi(t_i)\left\{%
\begin{array}{ll}
 \geq 1-\varepsilon~~~~~\mbox{if $i$ is odd},\\
\leq \varepsilon~~~~~~~~~~~~\mbox{if $i$ is  even}.
\end{array}%
\right.
\end{eqnarray*}
Moreover, let $\tau_0=0$ and $\tau_i$ denote the stopping time as
 \begin{eqnarray*}
\tau_i=\left\{%
\begin{array}{ll}
\min\{t>\tau_{i-1}:\varphi(t)\geq 1-\varepsilon\}~~~~~\mbox{if $i$ is odd}\\
\min\{t>\tau_{i-1}:\varphi(t)\leq \varepsilon\}~~~~~~~~~~~~\mbox{if $i$ is  even}
\end{array}%
\right.
\end{eqnarray*}
for $i\geq 1$, then for all $k\geq 0$ and $t\geq 0$,
\begin{eqnarray}\label{theo1_0}
\begin{aligned}
P\left(\tau_{2k+2}-\tau_{2k}> t \right) \leq (1-c)^{\lfloor t/T\rfloor},
\end{aligned}
\end{eqnarray}
where $c\in (0,1)$ and $T>0$ are constants depending on  $n, r_{\max},\eta, v$ and $\underline{\rho}$ only.
\end{thm}
\begin{proof}
First by Lemmas \ref{lem1} (or \ref{lem1_2}) and \ref{lem3} we can get  $ S_{\varepsilon}^2$ is finite-time robustly reachable from any initial state.
Also, define
\begin{eqnarray*}
\overline{ S}_{\varepsilon}:=\big\{(X,\theta)\in S^*: \frac{1}{n}\big\|\sum_{i=1}^n(\cos\theta_i,\sin\theta_i\big)\big\|\geq 1-\varepsilon\big\}.
\end{eqnarray*}
By Lemmas \ref{lem1} (or \ref{lem1_2}) and \ref{lem2} we have $\overline{ S}_{\varepsilon}$ is also finite-time robustly reachable for any initial state.
Using Lemma \ref{robust} our results can be obtained by taking $ S_1=\overline{ S}_{\varepsilon}$ and $ S_2= S_{\varepsilon}^2$.
\end{proof}

\begin{rem}
Compared to System II, the results for System I in Theorem \ref{result_1} (and also in Theorems \ref{result_2}, \ref{turn_1}, \ref{vortex_1}, \ref{result_3}, \ref{result_4}, \ref{turn_2} and \ref{vortex_2}) contain a condition $\eta>\frac{\pi}{2}-\frac{\pi}{n}$.
 This difference is a direct result of the difference between Lemmas \ref{lem1} and \ref{lem1_2}.
In fact, the condition $\eta>\frac{\pi}{2}-\frac{\pi}{n}$ for System I can be relaxed to $\eta>0$ under an assumption (Theorem \ref{relt_assmp}).
This assumption includes the assumption used by  Liu and Guo \cite{liu2009b}, who considered a consensus problem
for the original Vicsek model without noise.
\end{rem}

For any $\alpha>0$, similar to $S_{\alpha}^1$ we set
$$ S_{\alpha}^3:=\left\{(X,\theta)\in S^*: d_{\theta}<\alpha\right\}.$$
Differing from $S_{\alpha}^1$, $S_{\alpha}^3$ may not take the zero as its center angle.
Without any additional condition Systems I and II can reach a disordered state from an ordered state:
\begin{lem}\label{lem3_new}
Assume (A1) holds, then $( S_{\pi}^3)^c$ is finite-time robustly reachable from $S_{\pi}^3$ under both protocol  (\ref{model2_new}) and (\ref{model2}).
\end{lem}
The proof of this lemma is put in Appendix \ref{App_lem3_new}.

The following theorem says for any initial sate and system parameters the disordered states are still reached an infinite number of times:
\begin{thm}\label{result_org_1}
Assume (A1) holds. Then for System I (or II), with probability $1$ there exists an infinite time sequence $t_1<t_2<\cdots$  such that $d_{\theta(t_i)}\geq\pi$ for all $i\geq 1$;
moreover, let $\tau_0=0$ and $\tau_{i+1}$ denote the stopping time as
 \begin{eqnarray*}
\tau_{i+1}:=\min\{t>\tau_{i}:d_{\theta(t)}\geq \pi\},
\end{eqnarray*}
then for all $i\geq 1$ and $t\geq 0$,
\begin{eqnarray}\label{theo1_0}
\begin{aligned}
P\left(\tau_{i}-\tau_{i-1}> t \right)\leq (1-c)^{\lfloor t/T\rfloor},
\end{aligned}
\end{eqnarray}
where $c\in (0,1)$ and $T>0$ are constants depending on $n, r_{\max},\eta, v$ and  $\underline{\rho}$ only.
\end{thm}
\begin{proof}
Immediate from Corollary \ref{robust2} and Lemma \ref{lem3_new}.
\end{proof}

The possible applications and significance of Theorems \ref{result_1} and \ref{result_org_1} are provided in Section \ref{mrs_2} together with the corresponding results under periodic boundary conditions.

As mentioned in the Subsection \ref{Subsection_def2}, the robust consensus has been interested by many researchers \cite{Wang2009,Shi2013,Tian2009,Munz2011,Khoo2009,CY2011}. We also give a result for the robust consensus:
\begin{cor}\label{cor_1}
Assume (A1) holds,
then the robust consensus cannot be achieved for both Systems I and II.
\end{cor}
\begin{proof}
Immediate from Definition \ref{def_robust_consensus} and Theorem \ref{result_org_1}.
\end{proof}

Jadbabaie \emph{et al.} \cite{Jad1} analyzed System II without noise and found that to understand the effects of additive noise, one must focus on how noise affects the connectivity of the associated neighbor graphs. Later, Tahbaz-Salehi and Jadbabaie \cite{Jad2007} investigated the original Vicsek model without noise and found that the neighbor graphs are jointly connected over infinitely many time intervals for almost all initial states under periodic boundary conditions. The following Theorem provides an answer to how noise affects the connectivity under the open boundary conditions:
\begin{thm}\label{result_2}
Assume (A1) holds. Then
for System II (or System I with $\eta>\frac{\pi}{2}-\frac{\pi}{n}$),
$\{G(t)\}_{t=0}^{\infty}$ is not uniformly jointly weakly connected with probability $1$.
\end{thm}
The proof of this theorem which uses the idea appearing in the proof of Lemma \ref{lem3} is put in Appendix \ref{App_result_2}.

The colorful collective motion of animals has fascinated scientists from a wide array of fields. What exactly are the basic laws of collective motion, and how can they be understood empirically? Furthermore, what are the commonalities among the different factors in these laws? We established two theorems that concern turn, vortex, bifurcation, and merging:
\begin{thm}\label{turn_1}
Assume (A1) holds. Then
for System II (or System I with $\eta>\frac{\pi}{2}-\frac{\pi}{n}$), the events of turn, bifurcation and merging will happen an infinite number of times with probability $1$.
\end{thm}

\begin{thm}\label{vortex_1}
Assume (A1) holds. Then
for System I with $\eta>\frac{\pi}{2}-\frac{\pi}{n}$, with probability $1$ there exist vortices whose duration can be arbitrarily long.
\end{thm}
The proofs of Theorems \ref{turn_1} and \ref{vortex_1} are put in Appendix \ref{App_result_2}.

Our proposed method has favorable possible application in certain engineering systems.
For example, Yin, Wang, and Sun \cite{Yin2011,Wang2013} investigated some consensus algorithms for a platoon model, however there has been no crash analysis for them to date.
Using the idea of Lemma \ref{robust},
the crash analysis for these algorithms may be transformed to the design of cooperative controls such that the crash states are reached in finite time.
The method for the design of cooperative controls relates to the proofs of Theorems \ref{turn_1} and \ref{vortex_1}. Similarly, we can explore the design of collision avoidance algorithms for platoon model consensus via the proposed method.

\section{Results under Periodic Boundary Conditions}\label{mrs_2}

The system outlined by Vicsek \emph{et al.} \cite{Vicsek1995} assumes that all agents move in the square $[0,L)^2$ with periodic boundary conditions, suggesting that if an agent hits the boundary of the square, it will enter this square from the opposite boundary with the same velocity and heading.
 In mathematics, periodic boundary conditions contain two meanings:
(i) For all $i\in [1,n]$ and $t\geq 1$ we restrict $x_{i1}(t)$ and  $x_{i2}(t)$ to the interval $[0,L)$ by modulo $L$ when they are out of this interval; (ii) For all $i,j\in [1,n]$ and $t\geq 0$,
\begin{eqnarray*}
\begin{aligned}
&\|X_i(t)-X_j(t)\|_2^2\\
&=\min\{|x_{i1}(t)-x_{j1}(t)|, |x_{i1}(t)-x_{j1}(t)\pm L|\}^2\\
&~~+\min\{|x_{i2}(t)-x_{j2}(t)|, |x_{i2}(t)-x_{j2}(t)\pm L|\}^2.
\end{aligned}
\end{eqnarray*}

 Similar to Section \ref{mrs_1},
throughout this section we use the following assumption:

\textbf{(A2) }Assume that the population size $n\geq 2$, the parameters $\eta>0$, $v>0$, $r_i\geq 0$, $1\leq i\leq n$, and the initial positions $X(0)\in [0,L)^{2n}$ and headings $\theta(0)\in [-\pi,\pi)^n$ are arbitrarily given.  Also, assume
all the agents move in $[0,L)^2$ with periodic boundary conditions.



With the same proofs we can get that Lemmas \ref{lem1} and \ref{lem1_2} still hold under periodic boundary conditions.
Define $$r_{\max}:=\max_{1\leq i\leq n}r_i.$$
For Lemmas \ref{lem3} and \ref{lem3_new},
the corresponding versions under periodic boundary conditions are provided as follows:
\begin{lem}\label{lem4}
Assume (A2) is satisfied and let $\varepsilon>0$ be a constant arbitrarily given.
For both protocols (\ref{model2_new}) and (\ref{model2}), if
 \begin{eqnarray}\label{theo2_00}
L>\left\{%
\begin{array}{ll}
 2r_{\max}+2v\sum_{k=0}^{\lfloor \frac{\pi}{2\eta}-\frac{1}{2}\rfloor}\sin(\frac{\eta}{2}+k\eta)\\
 ~~\mbox{if $n$ is even or $\varepsilon> \frac{1}{n}$},\\
 3r_{\max}+2v\sum_{k=0}^{\lfloor \frac{\pi}{2\eta}+\frac{1}{\eta}\arcsin\frac{1}{n-1}-\frac{1}{2}\rfloor}\sin(\frac{\eta}{2}+k\eta)\\
 ~~\mbox{otherwise},\\
\end{array}%
\right.
\end{eqnarray}
then $S_{\varepsilon}^2$ is finite-time robustly reachable from $ S_{\eta}^1$.
\end{lem}

\begin{lem}\label{lem4_new}
Assume (A2) holds.
For both protocols (\ref{model2_new}) and (\ref{model2}), if
 \begin{eqnarray}\label{theo2_org_00}
L>2r_{\max}+2v\sum_{k=0}^{\lfloor \frac{\pi}{2\eta}-\frac{1}{2}\rfloor}\sin(\frac{\eta}{2}+k\eta),
\end{eqnarray}
then $( S_{\pi}^3)^c$ is finite-time robustly reachable from $ S_{\pi}^3$.
\end{lem}
The proofs of Lemmas \ref{lem4} and \ref{lem4_new} are put in Appendices \ref{App_lem4} and \ref{App_lem4_new} respectively.

Similar to Theorems \ref{result_1} and \ref{result_org_1} we give the following Theorems  \ref{result_3} and \ref{result_org_3}:
\begin{thm}\label{result_3}
Assume (A2) and (\ref{theo2_00}) hold,
then all the results of Theorem \ref{result_1} still hold with $c$ and $T$ depending on $L$ additionally.
\end{thm}
\begin{proof}
With the same proofs we can get that Lemmas \ref{lem1} and \ref{lem1_2} still hold under periodic boundary conditions.
With the same proof as Theorem \ref{result_1} but using Lemma \ref{lem4} instead of Lemma \ref{lem3} yields our result.
\end{proof}

\begin{thm}\label{result_org_3}
Assume (A2) and (\ref{theo2_org_00}) hold,
then all the results of Theorem \ref{result_org_1} still hold with $c$ and $T$ depending on $L$ additionally.
\end{thm}
\begin{proof}
Immediate from Corollary \ref{robust2} and Lemma \ref{lem4_new}.
\end{proof}

In the traditional sense, the order parameter of the SPP system has a phase transition with respect to noise and population density \cite{Vicsek1995,Czirak1999}; this requires an assumption that these systems will maintain order after a certain time, provided the noise is small and the population density is high. We would like to point out, however, that Theorems \ref{result_org_1}  and \ref{result_org_3} hold for any $\eta>0$ (providing (\ref{theo2_org_00}) holds under periodic boundary conditions), and $n\geq 2$, so for any noise intensity and population density, the order of the SPP system can be broken after a sufficient amount of time. Additionally, according to Theorems \ref{result_1}, \ref{result_3}, and the following Theorem \ref{relt_assmp}, the SPP system will switch between ordered and disordered states an infinite number of times for any noise amplitude and population density.
Thus, our results indicate that the order parameter does not exhibit the simple phase transition described in the literature \cite{Vicsek1995,Czirak1999}. Combining the results of our previous work \cite{Chen2014}, this allows us to deduce that the time interval between ordered and disordered states may exhibit a phase transition concerning noise and population density. Our results also provide mathematical proof of the observation that randomness can make non-equilibrium systems exhibit anomalously large fluctuations, which is true of many real-world systems such as glassy systems, granular packings, and active colloids \cite{Tsimring2014,Keizer1987}.

Similar to Theorem \ref{result_2}, we provide a theorem on how noise affects the connectivity under periodic boundary conditions.
\begin{thm}\label{result_4}
Assume (A2) holds.
For System II (or System I with $\eta>\frac{\pi}{2}-\frac{\pi}{n}$), if $L>2r_{\max}$
then $\{G(t)\}_{t=0}^{\infty}$ is not uniformly jointly weakly connected with probability $1$.
\end{thm}
The proof of this Theorem is put in Appendix \ref{App_result_4}.

By applying Theorem \ref{result_4} to the homogeneous case, it becomes clear that $\{G(t)\}_{t=0}^{\infty}$ is not uniformly jointly connected with Probability $1$. Corollary  \ref{cor_1} asserts that robust consensus cannot be reached under open boundary conditions, however, for the periodic boundaries, it remains unclear whether robust consensus can be reached. For systems whose network topologies are undirected and do not couple with their states, the uniformly joint connectivity of the network topologies is a necessary and sufficient condition for robust consensus \cite{Wang2009}. This condition is not applicable to our model, however.

Finally we give the corresponding results of Theorems \ref{turn_1} and \ref{vortex_1} for periodic boundary conditions.
\begin{thm}\label{turn_2}
Assume (A2) holds. Then
for System II (or System I with $\eta>\frac{\pi}{2}-\frac{\pi}{n}$),  with probability $1$ the event of turn will happen an infinite number of times for any $L>0$. Additionally, if (\ref{theo2_org_00}) is satisfied,
the events of bifurcation and merging will also happen an infinite number of times with probability $1$.
\end{thm}

\begin{thm}\label{vortex_2}
Assume (A2) holds.
If $\eta>\frac{\pi}{2}-\frac{\pi}{n}$ and $L>0$, then with probability $1$ System I will product vortices whose duration can be arbitrarily long.
\end{thm}
The proofs of Theorems \ref{turn_2} and \ref{vortex_2} are put in Appendix \ref{App_result_4}.

Buhl \emph{et al.} \cite{Buhl2006} used a one-dimensional version of the Vicsek model to investigate the collective behavior of locusts. By simulation, they found that the system exhibited large fluctuations of the order parameter and repeated changes in group's moving direction when the density of the individuals was low or average, but that the system became highly ordered after a short time when density was high. They also identified similarities between their simulations and real-world locust behavior. Because the homogeneous versions of Systems I and II have rules and features similar to the model in \cite{Buhl2006}, to some degree, Theorems  \ref{result_3}, \ref{turn_2}, \ref{vortex_2} and \ref{relt_assmp} can explain the repeated switches of the group's moving direction and the large fluctuations of the order parameter for low and medium densities - further, they allow us to predict that these behaviors still exist for high densities when the time step is sufficiently large.

\section{Results under An Assumption}\label{mrs_3}

Naturally, the original Vicsek model can evolve from disordered to ordered states; this has been verified through simulation \cite{Vicsek1995}, however, we can only prove it for the case of $\eta>\frac{\pi}{2}-\frac{\pi}{n}$  (which should also hold for any $\eta>0$). If its proof for $\eta>0$ became possible, Theorems \ref{result_1}, \ref{result_2}, \ref{turn_1}, \ref{vortex_1}, \ref{result_3}, \ref{result_4}, \ref{turn_2}, and \ref{vortex_2} would still hold in System I after relaxing the condition $\eta>\frac{\pi}{2}-\frac{\pi}{n}$ to $\eta>0$. This fact can be formulated as the following theorem:

\begin{thm}\label{relt_assmp}
For System I,
assume that (A1) (or (A2)) holds and there exists a finite time
$T>0$ depending on the system parameters only such that with
a positive probability $\min_{1\leq t\leq T} d_{\theta(t)}<\pi$.
Then the results about System I in Theorems \ref{result_1}, \ref{result_2}, \ref{turn_1}, \ref{vortex_1} (or \ref{result_3}, \ref{result_4}, \ref{turn_2} and \ref{vortex_2}) also hold when $\eta>\frac{\pi}{2}-\frac{\pi}{n}$ is replaced by $\eta>0$.
\end{thm}
The proof of this theorem is put in Appendix \ref{App_result_4}.

\section{Simulations}\label{sim}

To illustrate the relation between order parameters and population density or noise intensity, this section provides simulations of Systems I and II under periodic boundary conditions. All the following simulations assume the agents' speed is $v=0.01$, and that their initial headings and positions are independently and uniformly distributed in  $[-\pi,\pi)$ and $[0,5)^2$, respectively.
For $1\leq i,j\leq n$ and $t\geq 0$ we set the interaction weight
\begin{eqnarray*}
f_{ij}(t)=\left\{%
\begin{array}{ll}
1~~~~~~~~~~\mbox{if $\|X_i(t)-X_j(t)\|_2\leq r_i$},\\
0~~~~~~~~~~\mbox{else}.
\end{array}%
\right.
\end{eqnarray*}
Here we recall that $r_i$ is the interaction radius of agent $i$.

First, let's look at the order parameters of homogeneous Systems I and II under different population densities. In these simulations, we assumed the interaction radii of all agents are equal to $1$, and that noises $\{\xi_i(t)\}_{1\leq i\leq n, t\geq 0}$ are independently and uniformly distributed in $[-0.6,0.6]$. The maximum time step was set to  $10^6$. Figures  \ref{Fig1} and \ref{Fig2} show the values of the order function  $\varphi(t)$ of Systems I and II with $n = 10, 25$, and $40$, which represent low, medium, and high densities, respectively. As shown, from low to high density, Systems I and II exhibit ordered state at some moments and disordered state at other moments when the time grows large; this observation conforms entirely to our theoretical results for System II (Theorem \ref{result_3}), and implies that the condition $\eta>\frac{\pi}{2}-\frac{\pi}{n}$ for System I can be relaxed.

\begin{figure*}[htbp]
\begin{minipage}[t]{0.45\linewidth}
\centering
\includegraphics[width=3in]{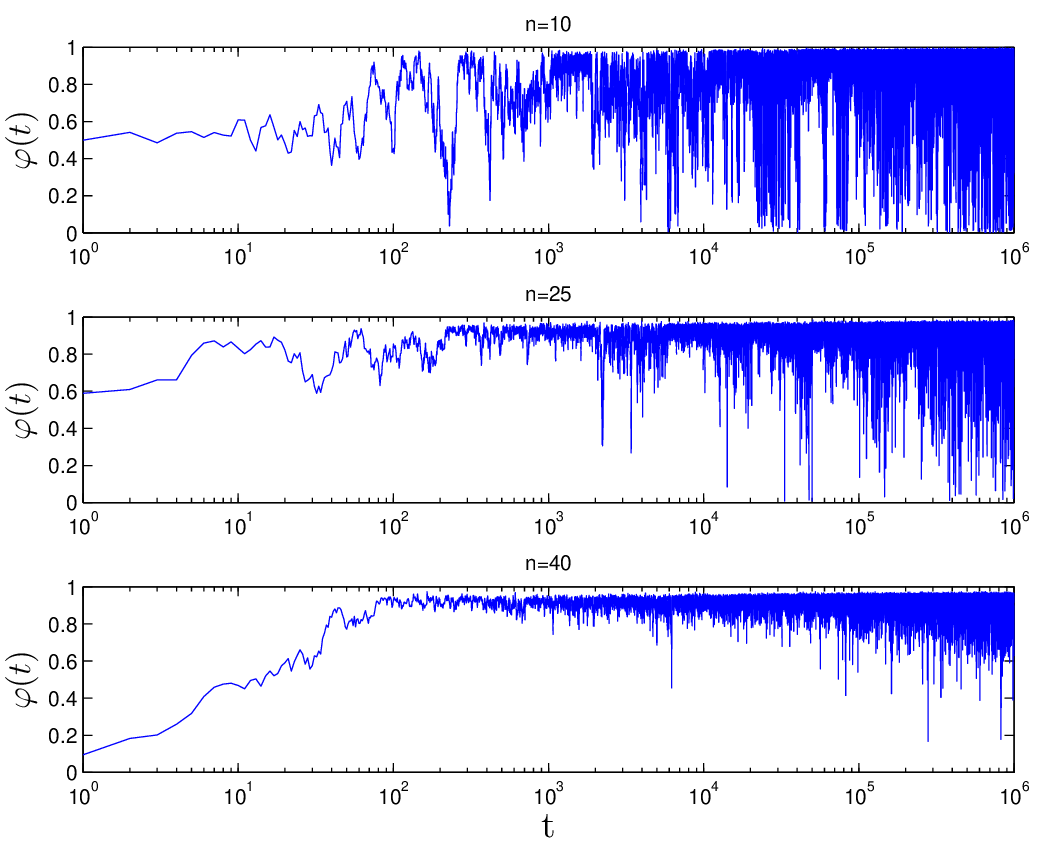}
\caption{The order parameter  $\varphi(t)$  of homogeneous System I (original Vicsek model) with $n=10,25,40$.}\label{Fig1}
\end{minipage}
\hfill
\begin{minipage}[t]{0.45\linewidth}
\centering
\includegraphics[width=3in]{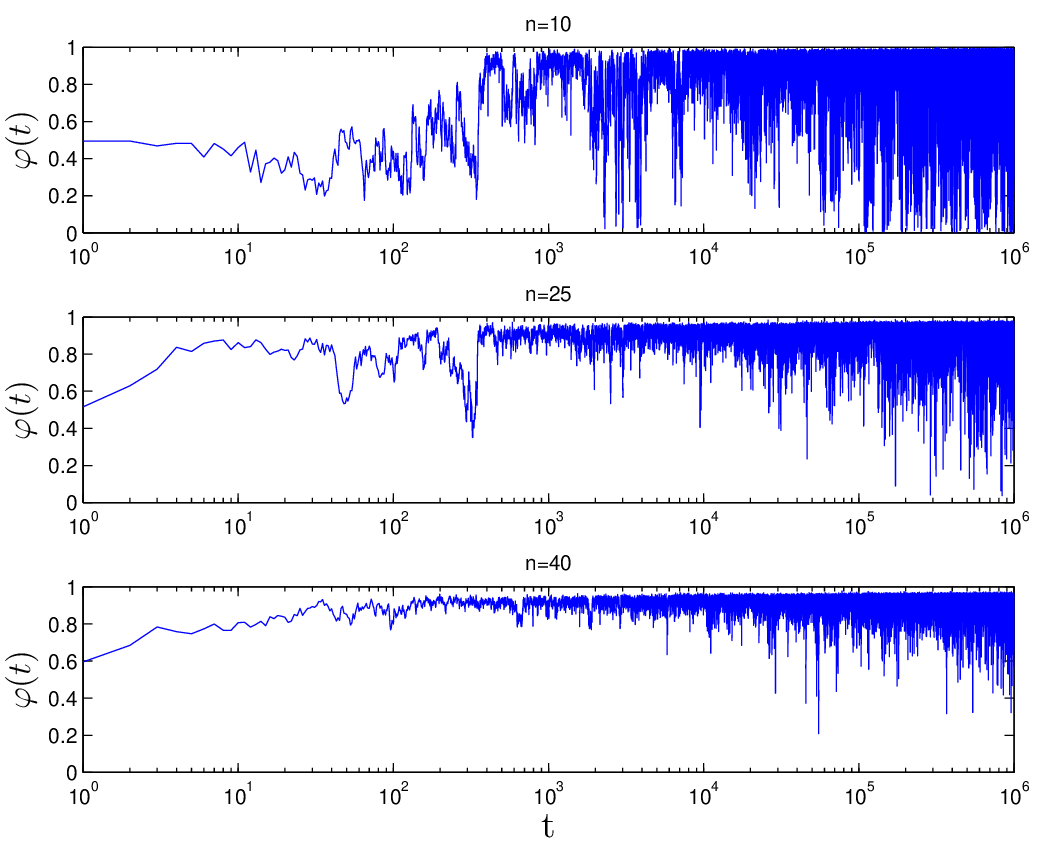}
\caption{The order parameter $\varphi(t)$ of  homogeneous System II with $n=10,25,40$.}\label{Fig2}
\end{minipage}
\end{figure*}

 \begin{figure*}[htbp]
\begin{minipage}[t]{0.45\linewidth}
\centering
\includegraphics[width=3in]{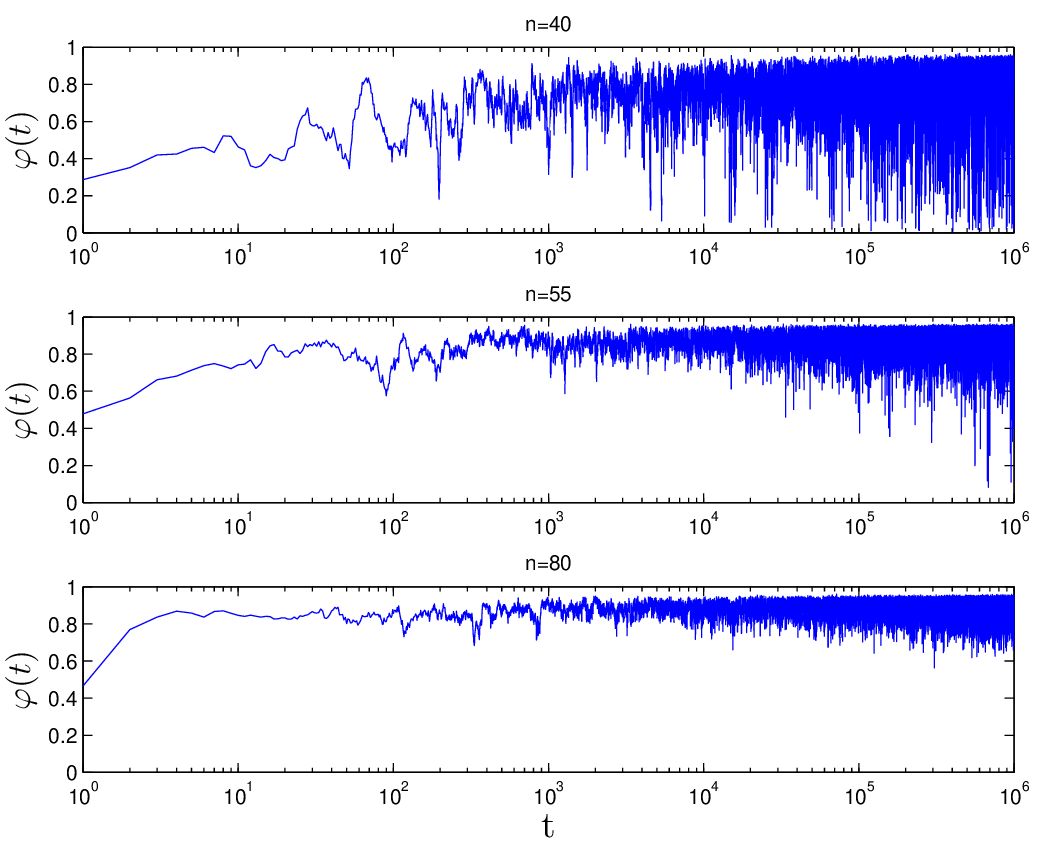}
\caption{The order parameter  $\varphi(t)$  of heterogeneous System I with $n=40,55,80$.}\label{Fig3}
\end{minipage}
\hfill
\begin{minipage}[t]{0.45\linewidth}
\centering
\includegraphics[width=3in]{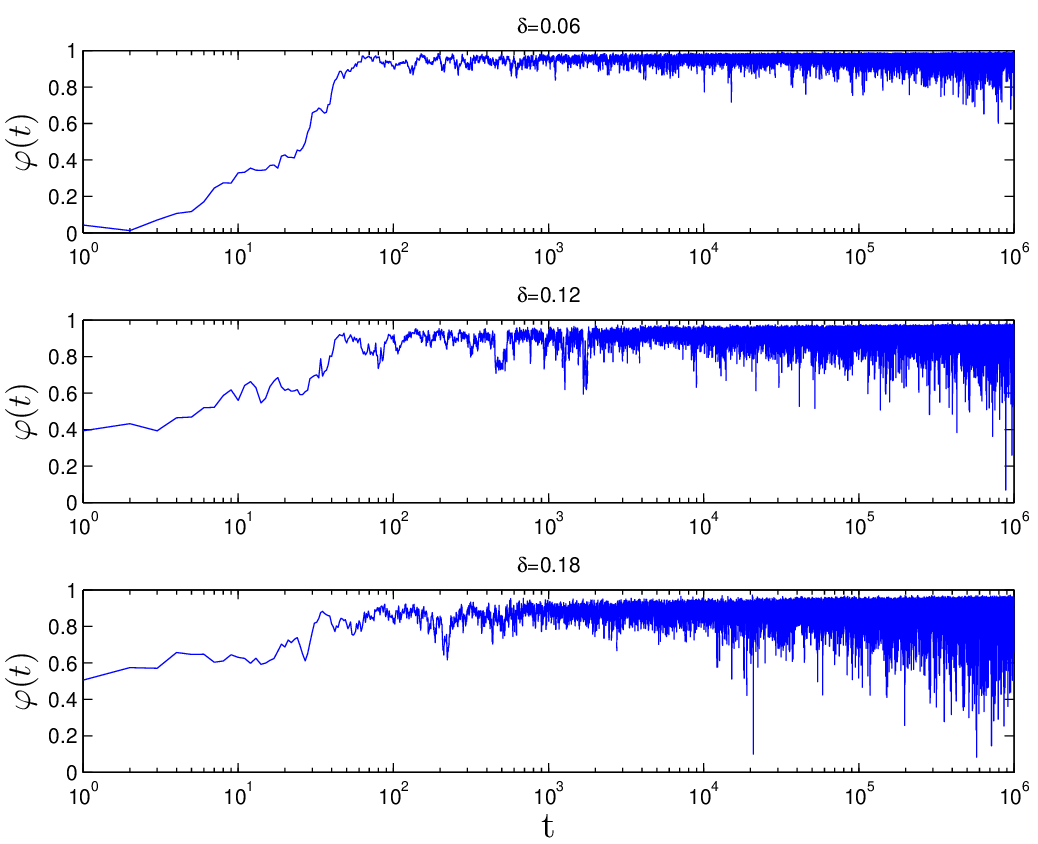}
\caption{The order parameter  $\varphi(t)$ of homogeneous System I under Gaussian noise with variance $\delta=0.06,0.12,0.18$.}\label{Fig4}
\end{minipage}
\end{figure*}

We also simulated a heterogeneous System I by assuming that the interaction radius of each agent is independently and uniformly distributed in $[0,2]$. With the same noise assumption shown in Figure \ref{Fig1}, Figure \ref{Fig3} shows the value of the order parameter of heterogeneous System I with $n=40, 55,$ and $80$. By comparing Figure \ref{Fig1} to Figure \ref{Fig3} with the same population size ($n=40$), it seems that homogeneity benefits the order of the system rather than heterogeneity. In addition, Figures \ref{Fig1}-\ref{Fig3} show that for both Systems I and II, higher population density reduces order parameter fluctuation.

Finally, we simulated a homogenous System I by assuming $\{\xi_i(t)\}_{1\leq i\leq n, t\geq 0}$ to an i.i.d. zero-mean Gaussian noise sequence. To investigate the influence of different noise intensities on the order parameter, we set the variance $\delta$ of the noise to be $0.06$, $0.12$, or $0.18$ (Figure  \ref{Fig4}). We found that larger noise intensity increases order parameter fluctuation.

\section{Conclusion}\label{conclude}
Self-organized systems characterized by deterministic laws and randomness commonly exist in real-world natural, engineering, social, and economic systems. Accurate analysis of the local rules of these systems as they affect their global behavior is a common (and quite challenging) problem in many fields.  In this paper, we proposed an innovative, general approach to this problem that transforms it to the design of cooperative control algorithms. Using our method, we revealed the manner in which noise affects the order and connectivity of heterogeneous SPP systems, and also showed that these systems can spontaneously produce turn, vortex, bifurcation, and flock merging phenomena.

An interesting problem inherent to the SPP system is minimizing the effects of noise to keep the system in order. A possible method of doing so is to adopt the distributed stochastic approximation, under which each agent uses a decreasing gain function acting on its neighbors' information to reduce measurement or communication noise \cite{Li2010,Huang2012,Yin2011,Cheng2014}.

Unfortunately, as many researchers have pointed out, the Vicsek model is very basic but probably not particularly descriptive of actual biological clusters. In the future, we plan to use our proposed method to analyze more practical systems. Of course, the design of control algorithms remains challenging in terms of complex, real-world systems. Another attractive future research direction is the development of corresponding theories for designing these types of algorithms.

\section*{Acknowledgment}
The author would like to thank the guidance of Prof. Lei Guo from
Academy of Mathematics and Systems Science, Chinese Academy of Sciences.

\appendices
\section{Proof of Lemma \ref{lem3} }\label{App_lem3}
We first consider the protocol (\ref{model2}).
Without loss of generality we assume $\varepsilon\in(0,1)$. Define the constant $\beta:=\min\{\frac{\eta}{2},2\arcsin\frac{\varepsilon}{2}\}$.
we will prove our result for the following two cases respectively:\\
\textbf{Case I}: $n$ is even.
 We separate the $n$ agents into two disjoint sets $\mathcal{A}_1$ and $\mathcal{A}_2$ with $|\mathcal{A}_1|=|\mathcal{A}_2|=\frac{n}{2}$,
and $x_{i2}(0)\geq x_{j2}(0)$ for any agent $i\in \mathcal{A}_1, j\in \mathcal{A}_2$. Here we recall that
 $x_{i2}(0)$ denotes the second coordinate of $X_i(0)$. Let
 \begin{eqnarray}\label{lem3_1_1}
 t_1:=\big\lfloor \frac{r_{\max}}{2v\sin(\eta/4)}\big\rfloor+1.
 \end{eqnarray}
 For $0\leq t<t_1$, we choose
 \begin{eqnarray}\label{lem3_1_a}
\delta_i(t)=\frac{\eta}{8},~~~~\forall 1\leq i\leq n,
\end{eqnarray}
 and set
  \begin{eqnarray}\label{lem3_1_b}
u_i(t)=\left\{%
\begin{array}{ll}
\frac{3\eta}{8}-\widetilde{\theta}_i(t)~~~~~~~\mbox{if}~i\in\mathcal{A}_1,\\
-\frac{3\eta}{8}-\widetilde{\theta}_i(t)~~~~\mbox{if}~i\in\mathcal{A}_2.
\end{array}%
\right.
\end{eqnarray}
From this we can get for all $t\in[0,t_1)$,
  \begin{eqnarray}\label{lem3_1}
\theta_i(t+1)\in\left\{%
\begin{array}{ll}
 [\eta/4,\eta/2]~~~~~~~~\mbox{if}~i\in\mathcal{A}_1,\\
\normalsize[-\eta/2,-\eta/4\normalsize]~~\mbox{if}~i\in\mathcal{A}_2.
\end{array}%
\right.
\end{eqnarray}
From this  for any $i\in\mathcal{A}_1$ and $j\in\mathcal{A}_2$, we have
\begin{eqnarray}\label{lem3_2}
&&x_{i2}(t_1)-x_{j2}(t_1)\nonumber\\
&&=x_{i2}(0)+\sum_{0<t\leq t_1}v\sin\theta_i(t)-x_{j2}(0)-\sum_{0<t\leq t_1}v\sin\theta_j(t)\nonumber\\
&&\geq v \sum_{0<t\leq t_1}2\sin\frac{\eta}{4}=2vt_1\sin\frac{\eta}{4}>r_{\max}.
\end{eqnarray}
which indicates that there exists no edge between  $\mathcal{A}_1$ and $\mathcal{A}_2$ at time $t_1$.
Also, by (\ref{lem3_1_b}), (\ref{lem3_1}) and the condition of $\max_{1\leq i\leq n}|\theta_i(0)|\leq \frac{\eta}{2}$ we have
\begin{eqnarray}\label{lem3_1_2}
u_i(t)\in\left[-\frac{7\eta}{8},\frac{7\eta}{8}\right]=\left[-\eta+\delta_i(t),\eta-\delta_i(t)\right]
\end{eqnarray}
for any $1\leq i\leq n$ and $0\leq t<t_1$.

Next we will give a control algorithm to minimize the value of the order function. Set $$t_2:=\max\left\{t_1+\big\lceil\frac{\pi-2\beta}{\eta}-\frac{1}{2} \big\rceil, t_1+1\right\}.$$
For $t\in[t_1,t_2)$, we choose
\begin{eqnarray}\label{lem3_3}
(\delta_i(t),u_i(t))=\left\{%
\begin{array}{ll}
\big(\frac{\eta}{4},\frac{3\eta}{4}\big)~~~~~\mbox{if}~\widetilde{\theta}_i(t)<\pi/2+\beta-\eta\\
\big(\beta, \frac{\pi}{2}-\widetilde{\theta}_i(t)\big)~~\mbox{otherwise}
\end{array}%
\right.
\end{eqnarray}
for $i\in\mathcal{A}_1$, and choose
\begin{eqnarray}\label{lem3_4}
(\delta_i(t),u_i(t))=\left\{%
\begin{array}{ll}
\big(\frac{\eta}{4},-\frac{3\eta}{4}\big)~~\mbox{if}~\widetilde{\theta}_i(t)>-\pi/2-\beta+\eta\\
\big(\beta, -\frac{\pi}{2}-\widetilde{\theta}_i(t)\big)~~\mbox{otherwise}
\end{array}%
\right.
\end{eqnarray}
for $i\in\mathcal{A}_2$. From (\ref{lem3_1}), (\ref{lem3_3}) and (\ref{lem3_4}) it can be computed that
\begin{eqnarray}\label{lem3_4_a}
u_i(t)\in [-\eta+\delta_i(t),\eta-\delta_i(t)],~\forall 1\leq i\leq n, t_1\leq t<t_2.
\end{eqnarray}
If the sets $\mathcal{A}_1$ and $\mathcal{A}_2$ are disconnected at time $t$, then with the similar methods to (\ref{lem1_3}) and (\ref{lem1_5}) we can get
\begin{eqnarray*}
\min_{i\in\mathcal{A}_1}\theta_i(t+1)\left\{%
\begin{array}{ll}
\in\big[\frac{\pi}{2}-\beta,\frac{\pi}{2}+\beta\big]~\mbox{if}~\min\limits_{i\in\mathcal{A}_1}\theta_i(t)\geq \frac{\pi}{2}-\beta-\frac{\eta}{2},\\
\geq \frac{\eta}{2}+\min\limits_{i\in\mathcal{A}_1}\theta_i(t)~~\mbox{otherwise},
\end{array}%
\right.
\end{eqnarray*}
and
\begin{eqnarray*}
\max_{i\in\mathcal{A}_2}\theta_i(t+1)\left\{%
\begin{array}{ll}
\in\big[-\frac{\pi}{2}-\beta,-\frac{\pi}{2}+\beta\big]\\
~~~\mbox{if}~\max\limits_{i\in\mathcal{A}_2}\theta_i(t)\leq -\frac{\pi}{2}+\beta+\frac{\eta}{2},\\
\leq -\frac{\eta}{2}+\max\limits_{i\in\mathcal{A}_2}\theta_i(t)~~\mbox{otherwise}.
\end{array}%
\right.
\end{eqnarray*}
So by (\ref{lem3_2}) and induction we can get $\mathcal{A}_1$ and $\mathcal{A}_2$ are always disconnected in the time $[t_1,t_2]$. Then, similar to (\ref{lem1_8}) we have
 \begin{eqnarray}\label{lem3_5}
\theta_i(t_2)\in\left\{%
\begin{array}{ll}
 [\pi/2-\beta,\pi/2+\beta]~~~~~\mbox{if}~i\in\mathcal{A}_1,\\
\normalsize[-\pi/2-\beta,-\pi/2+\beta\normalsize]~~\mbox{if}~i\in\mathcal{A}_2,
\end{array}%
\right.
\end{eqnarray}
which is followed by
\begin{eqnarray}\label{lem3_6}
\begin{aligned}
\varphi(t_2)&=\frac{1}{n}\big\|\sum_{i\in\mathcal{A}_1\cup\mathcal{A}_2}\big(\cos\theta_i(t_2), \sin\theta_i(t_2)  \big)\big\|\\
&=\frac{1}{n}\big\|\sum_{i\in\mathcal{A}_1}\big(\cos\theta_i(t_2), \sin\theta_i(t_2)-1  \big)+\sum_{i\in\mathcal{A}_2}\big(\cos\theta_i(t_2), \sin\theta_i(t_2)+1  \big)\big\|\\
&\leq \big\|\big(\cos\big(\frac{\pi}{2}-\beta\big), \sin\big(\frac{\pi}{2}-\beta\big)-1 \big)\big\|\\
&=\sqrt{2-2\cos\beta}= 2\sin\frac{\beta}{2}\leq\varepsilon.
\end{aligned}
\end{eqnarray}
Together this with (\ref{lem3_1_2}) and (\ref{lem3_4_a}) we have $ S_{\varepsilon}^2$ is robustly reachable at time $t_2$. \\
\textbf{Case II}: $n$ is odd. We separate the $n$ agents into three disjoint sets $\mathcal{A}_1$, $\mathcal{A}_2$ and $\mathcal{A}_3$ which satisfy that $|\mathcal{A}_1|=|\mathcal{A}_2|=\frac{n-1}{2}, |\mathcal{A}_3|=1$,  and $x_{i2}(0)\geq x_{j2}(0) \geq x_{k2}(0)$ for any agent $i\in \mathcal{A}_1, j\in \mathcal{A}_3$ and $k\in \mathcal{A}_2$.
Let
$$t_3:=\big\lfloor \frac{r_{\max}}{v(\sin\frac{\eta}{4}-\sin\frac{\eta}{8})}\big\rfloor+1.$$
 For $0\leq t<t_3$, we choose $\delta_i(t)=\frac{\eta}{8}$, $1\leq i\leq n$, and set
  \begin{eqnarray}\label{lem3_7}
u_i(t)=\left\{%
\begin{array}{ll}
\frac{3\eta}{8}-\widetilde{\theta}_i(t)~~~~~~~\mbox{if}~i\in\mathcal{A}_1,\\
-\frac{3\eta}{8}-\widetilde{\theta}_i(t)~~~~\mbox{if}~i\in\mathcal{A}_2,\\
-\widetilde{\theta}_i(t)~~~~\mbox{if}~i\in\mathcal{A}_3.
\end{array}%
\right.
\end{eqnarray}
Similar to (\ref{lem3_1}) and (\ref{lem3_1_2}), we can get for all $t\in[0,t_3)$,
 \begin{eqnarray*}\label{lem3_13}
\theta_i(t+1)\in\left\{%
\begin{array}{ll}
 [\eta/4,\eta/2]~~~~~\mbox{if}~i\in\mathcal{A}_1,\\
\normalsize[-\eta/2,-\eta/4\normalsize]~~\mbox{if}~i\in\mathcal{A}_2,\\
\normalsize[-\eta/8, \eta/8\normalsize]~~\mbox{if}~i\in\mathcal{A}_3,
\end{array}%
\right.
\end{eqnarray*}
and
\begin{eqnarray}\label{lem3_13_a}
 u_i(t)\in[-\eta+\delta_i(t),\eta-\delta_i(t)], ~~\forall 1\leq i\leq n.
 \end{eqnarray}
Also, similar to (\ref{lem3_2}) we can get
the sets $\mathcal{A}_1$, $\mathcal{A}_2$ and $\mathcal{A}_3$ are mutually disconnected at time $t_3$.

Let $c_n:=\frac{\pi}{2}+\arcsin\frac{1}{n-1}$ and set $$t_4:=\max\left\{t_3+\big\lceil\frac{2c_n-2\beta}{\eta}-\frac{1}{2} \big\rceil, t_3+1\right\}.$$
For all $t\in[t_3,t_4)$, similar to (\ref{lem3_3}) and (\ref{lem3_4}) we choose
\begin{eqnarray*}\label{lem3_14}
(\delta_i(t),u_i(t))=\left\{%
\begin{array}{ll}
\big(\frac{\eta}{4},\frac{3\eta}{4}\big)~~~~~\mbox{if}~\widetilde{\theta}_i(t)<c_n+\beta-\eta\\
\big(\beta, c_n-\widetilde{\theta}_i(t)\big)~~\mbox{otherwise}
\end{array}%
\right.
\end{eqnarray*}
for $i\in\mathcal{A}_1$, and choose
\begin{eqnarray*}\label{lem3_15}
(\delta_i(t),u_i(t))=\left\{%
\begin{array}{ll}
\big(\frac{\eta}{4},-\frac{3\eta}{4}\big)~~~~~\mbox{if}~\widetilde{\theta}_i(t)>-c_n-\beta+\eta\\
\big(\beta, -c_n-\widetilde{\theta}_i(t)\big)~~\mbox{otherwise}
\end{array}%
\right.
\end{eqnarray*}
for $i\in\mathcal{A}_2$.
Also, for $i\in\mathcal{A}_3$,
set $\delta_i(t)=\beta$ and $u_i(t)=-\widetilde{\theta}_i(t)$. Similar to (\ref{lem3_4_a}) we can get
\begin{eqnarray}\label{lem3_15_a}
u_i(t)\in [-\eta+\delta_i(t),\eta-\delta_i(t)],~~\forall 1\leq i\leq n, t_3\leq t<t_4.
\end{eqnarray}
Also, similar to Case I we have
$\mathcal{A}_1$, $\mathcal{A}_2$ and $\mathcal{A}_3$ are always mutually disconnected in the time $[t_3,t_4]$.  Thus,  similar to (\ref{lem3_5}) we can get
 \begin{eqnarray}\label{lem3_16}
\theta_i(t_4)\in\left\{%
\begin{array}{ll}
 [c_n-\beta,c_n+\beta]~~~~~\mbox{if}~i\in\mathcal{A}_1\\
\normalsize[-c_n-\beta,-c_n+\beta\normalsize]~~\mbox{if}~i\in\mathcal{A}_2\\
\normalsize[-\beta,\beta\normalsize]~~\mbox{if}~i\in\mathcal{A}_3
\end{array}%
\right.,
\end{eqnarray}
which indicates that
\begin{eqnarray}\label{lem3_17}
\begin{aligned}
\varphi(t_4)&=\frac{1}{n}\big\|\sum_{i\in\mathcal{A}_1}\big(\cos\theta_i(t_4)-\cos c_n, \sin\theta_i(t_4)-\sin c_n  \big)\\
&~~~~+\sum_{i\in\mathcal{A}_2}\big(\cos\theta_i(t_4)-\cos c_n, \sin\theta_i(t_4)+\sin c_n  \big)\\
&~~~~+\sum_{i\in\mathcal{A}_3}\big(\cos\theta_i(t_4)-1, \sin\theta_i(t_4)\big)\big\|\\
&\leq \sqrt{2-2\cos\beta}= 2\sin\frac{\beta}{2}\leq\varepsilon.
\end{aligned}
\end{eqnarray}
Together this with (\ref{lem3_13_a}) and (\ref{lem3_15_a}) we have $ S_{\varepsilon}^2$ is robustly reachable at time $t_4$.

For protocol (\ref{model2_new}), if $\eta<\pi/2$ we can get our result with the similar method as protocol (\ref{model2}). Otherwise,
by Lemma \ref{lem1_2} we can control the state of the system to $S_{\eta'}^1$ with $\eta'<\pi/2$, then with the similar method as protocol (\ref{model2}) yields our result.

\section{Proof of Lemma \ref{lem3_new}}\label{App_lem3_new}

We will discuss protocol (\ref{model2_new}) first.
Because the System I has the isotropic property under open boundary conditions,  without loss of generality we assume the initial headings $\theta_i(0)$, $1\leq i\leq n$ are distributed in the interval $[-\pi/2,\pi/2)$. Thus we can get
\begin{eqnarray}\label{lem3_new_0}
\theta_{\min}(0)\leq \widetilde{\theta}_i(0)\leq \theta_{\max}(0),~~\forall 1\leq i\leq n.
\end{eqnarray}
For $t\geq 0$ and $1\leq i \leq n$, we choose $(\delta_i(t),u_i(t))$ as same as (\ref{lem1_1}) but using $\eta$ instead of $\alpha$. With almost the same process of (\ref{lem1_2})-(\ref{lem1_8})  we have
\begin{eqnarray}\label{lem3_new_1}
\max_{1\leq i \leq n}|\theta_i(t_1')|\leq \eta/2,
\end{eqnarray}
where $t_1':=\lceil \frac{\pi}{\eta}\rceil$-1.

Similar to the case II of the proof of Lemma \ref{lem3}, we separate the $n$ agents into three non-empty disjoint sets $\mathcal{A}_1$, $\mathcal{A}_2$ and $\mathcal{A}_3$ with $x_{i2}(t_1')\geq x_{j2}(t_1') \geq x_{k2}(t_1')$ for any agent $i\in \mathcal{A}_1, j\in \mathcal{A}_3$ and $k\in \mathcal{A}_2$.
Let
$$t_3':=t_1'+\big\lfloor \frac{r_{\max}}{v(\sin\frac{\eta}{4}-\sin\frac{\eta}{8})}\big\rfloor+1.$$
 For $t_1'\leq t<t_3'$, we choose $\delta_i(t)=\frac{\eta}{8}$, $1\leq i\leq n$, and set $u_i(t)$ as same as (\ref{lem3_7}).
With the similar discussion to the case II of the proof of Lemma \ref{lem3}  we have
the sets $\mathcal{A}_1$, $\mathcal{A}_2$ and $\mathcal{A}_3$ are mutually disconnected at time $t_3'$.

Let $t_4':=t_3'+\lceil\frac{6\pi}{5\eta}\rceil-1$.
For all $t\in[t_3',t_4')$ we set $\delta_i(t)=\eta/8$, and choose
\begin{eqnarray*}\label{lem3_new_2}
u_i(t)=\left\{%
\begin{array}{ll}
\frac{3\eta}{4}~~~~~\mbox{if}~\widetilde{\theta}_i(t)<\frac{3\pi}{4}-\frac{3\eta}{4}\\
\frac{3\pi}{4}-\widetilde{\theta}_i(t)~~\mbox{otherwise}
\end{array}%
\right.
\end{eqnarray*}
for $i\in\mathcal{A}_1$,
\begin{eqnarray*}\label{lem3_new_3}
u_i(t)=\left\{%
\begin{array}{ll}
-\frac{3\eta}{4}~~~~~\mbox{if}~\widetilde{\theta}_i(t)>-\frac{3\pi}{4}+\frac{3\eta}{4}\\
-\frac{3\pi}{4}-\widetilde{\theta}_i(t)~~\mbox{otherwise}
\end{array}%
\right.
\end{eqnarray*}
for $i\in\mathcal{A}_2$, $u_i(t)=-\widetilde{\theta}_i(t)$ for $i\in\mathcal{A}_3$. Similar to (\ref{lem3_16}) we can get
 \begin{eqnarray}\label{lem3_new_4}
\theta_i(t_4')\in\left\{%
\begin{array}{ll}
 [\frac{3\pi}{4}-\frac{\eta}{8},\frac{3\pi}{4}+\frac{\eta}{8}]~~~~~\mbox{if}~i\in\mathcal{A}_1\\
\normalsize[-\frac{3\pi}{4}-\frac{\eta}{8},-\frac{3\pi}{4}+\frac{\eta}{8}\normalsize]~~\mbox{if}~i\in\mathcal{A}_2\\
\normalsize[-\frac{\eta}{8},\frac{\eta}{8}\normalsize]~~\mbox{if}~i\in\mathcal{A}_3
\end{array}%
\right..
\end{eqnarray}
With the fact of $\eta\in(0,\pi)$ we have $d_{\theta(t_4')}>\pi$.

For protocol (\ref{model2}), with the same process as (\ref{lem3_new_0})-(\ref{lem3_new_4}) our result follows.

\section{Proofs of Theorems \ref{result_2},\ref{turn_1} and \ref{vortex_1}}\label{App_result_2}

\begin{proof} \textbf{of Theorem \ref{result_2}}
We first consider the System II.
For any $t\geq 0$, if $\max_{1\leq i \leq n}|\theta_i(t)|\leq \frac{\eta}{2}$, similar to the proof of Lemma \ref{lem3} we separate the $n$ agents into two disjoint sets $\mathcal{A}_1$ and $\mathcal{A}_2$ with  $|\mathcal{A}_1|=\lceil\frac{n}{2}\rceil$, $|\mathcal{A}_2|=\lfloor\frac{n}{2}\rfloor$,
and $x_{i2}(0)\geq x_{j2}(0)$ for any agent $i\in \mathcal{A}_1, j\in \mathcal{A}_2$. Let
$T_1:=\big\lfloor \frac{r_{\max}}{2v\sin(\eta/4)}\big\rfloor+1$
and $T$ be an arbitrary large integer. Under protocol (\ref{model2}), for $t'\in[t,t+T_1+T)$, we choose $\delta_i(t')$ and $u_i(t')$ as same as
(\ref{lem3_1_a}) and  (\ref{lem3_1_b}) respectively. Then by (\ref{lem3_1}) and (\ref{lem3_2}) we can get that
there is always no edge between $\mathcal{S}_1$ and $\mathcal{S}_2$ in the time $[t+T_1,t+T_1+T]$. With this and the method of
(\ref{rob_1}) we have for System II,
\begin{eqnarray}\label{result_2_1}
\begin{aligned}
&P\left(\bigcup_{t'=t+T_1}^{t+T_1+T}G(t') \mbox{ is not weakly connected}\big| \forall (X(t),\theta(t))\in S_{\eta}^1 \right)\geq \frac{1}{8^{n(T_1+T)}}.
\end{aligned}
\end{eqnarray}
Also, for any initial state, any $t\geq 0$ and $w_{t-1}\in \Omega^{t-1}$, together the proof process of Lemma \ref{lem1} and the method of (\ref{rob_1}) we can get for System II,
\begin{eqnarray}\label{result_2_2}
P\left((X(t+T_2),\theta(t+T_2))\in S_{\eta}^1 |w_{t-1}\right)\geq \frac{1}{4^{nT_2}},
\end{eqnarray}
where $T_2:=\lceil \frac{2(\pi-\eta/4)}{\eta}\rceil=\lceil \frac{2\pi}{\eta}-\frac{1}{4}\rceil$.
By (\ref{result_2_1}), (\ref{result_2_2}) and Bayes' theorem we have
\begin{eqnarray*}\label{result_2_3}
\begin{aligned}
&P\left(\bigcup_{t'=t+T_1+T_2}^{t+T_1+T_2+T}G(t')\mbox{ is not weakly connected}\big| w_{t-1} \right)\\
&\geq P\left((X(t+T_2),\theta(t+T_2))\in S_{\eta}^1 |w_{t-1}\right)\\
&\cdot P\left(\bigcup_{t'=t+T_1+T_2}^{t+T_1+T_2+T}G(t') \mbox{ is not weakly connected}\big| (X(t+T_2),\theta(t+T_2))\in S_{\eta}^1, w_{t-1}\right)\\
&\geq \frac{1}{4^{nT_2} 8^{n(T_1+T)}}.
\end{aligned}
\end{eqnarray*}
Similar to (\ref{rob_3}) with probability $1$ there is a time $t^*>0$ such that $\bigcup_{t'=t^*+T_1+T_2}^{t^*+T_1+T_2+T}G(t')$ is not weakly connected.

For System I, with the same process as above but using Lemma \ref{lem1_2} instead of Lemma \ref{lem1} we can get our result.
\end{proof}

\vskip 2mm

\begin{proof} \textbf{of Theorem \ref{turn_1}}
For any $t\geq 0$, if $\max_{1\leq i \leq n}|\theta_i(t)|\leq \frac{\varepsilon}{2}$ with
$\varepsilon$ being a small positive constant, then we
set $T:=\lceil\frac{(\pi-\varepsilon)K+\eta}{2(K-1)\eta}\rceil$ with $K$ being a large integer, and choose $\delta_i(t')=\frac{\eta}{2K}$ and
\begin{eqnarray*}
u_i(t')=\left\{%
\begin{array}{ll}
\eta-\frac{\eta}{2K}~~~~~\mbox{if}~\widetilde{\theta}_i(t')<\frac{\pi}{2}+\frac{\eta}{2K}-\eta\\
\frac{\pi}{2}-\widetilde{\theta}_i(t')~~~~\mbox{otherwise}
\end{array}%
\right.
\end{eqnarray*}
 for $t'\in [t,t+T)$ and $i=1,\ldots,n$. For System II (or I), under this process we can get that $\theta_i(t+T)\in [\frac{\pi}{2}-\frac{\eta}{2K},\frac{\pi}{2}+\frac{\eta}{2K}]$ for $1\leq i\leq n$, and during the time $[t,t+T)$ all the agents keep almost synchronization, which indicate the event of turn has happened. Using Lemmas \ref{robust} and \ref{lem1} (or \ref{lem1_2}) we can get for System II (or I with $\eta>\frac{\pi}{2}-\frac{\pi}{n}$), the event of turn will happen an infinite number of times with probability $1$.

Similarly, combing (\ref{lem3_5}), Lemmas \ref{robust} and \ref{lem1} (or \ref{lem1_2}) we can get for System II (or I with $\eta>\frac{\pi}{2}-\frac{\pi}{n}$), the events of bifurcation and merging will happen an infinite number of times with probability $1$.
\end{proof}

\begin{proof}\textbf{ of Theorem \ref{vortex_1}}
Because System I has the property of isotropy, we can get it exists vortices with arbitrarily long duration by adding the turning angles in the proof of the turn event of Theorems \ref{turn_1}.
\end{proof}

\section{Proof of Lemma \ref{lem4}}\label{App_lem4}

We consider protocol (\ref{model2}) first.
This proof partly takes the ideas of the proof of Lemma \ref{lem3}.
Given a large integer $K>0$, throughout this proof we choose $\delta_i(t)=\frac{\eta}{2K}$  for $i=1,\ldots,n$ and $t\geq 0$.
Set $t_0:=\lceil \frac{L}{2v\sin \frac{\eta}{K}} \rceil$. For $i=1,\ldots,n$ and $t\in[0,t_0)$, we choose
\begin{eqnarray}\label{lem4_1}
u_i(t)=\left\{%
\begin{array}{ll}
-\frac{3\eta}{2K}-\widetilde{\theta}_i(t)~~~~~\mbox{if}~x_{i2}(t)\in \big[\frac{L}{2},L\big),\\
\frac{3\eta}{2K}-\widetilde{\theta}_i(t)~~~~~~\mbox{if}~x_{i2}(t)\in\big[0,\frac{L}{2}\big).
\end{array}%
\right.
\end{eqnarray}
Under protocol (\ref{model2}),  for $t\in[0,t_0)$,
in the case of $x_{i2}(t)\geq L/2$, we have $\theta_i(t+1)\in [-2\eta/K,-\eta/K]$ and
 \begin{eqnarray*}\label{lem4_2}
 \begin{aligned}
 x_{i2}(t+1)&= x_{i2}(t)+v\sin\theta_i(t+1)\in [x_{i2}(t)-v\sin \frac{2\eta}{K},x_{i2}(t)-v\sin \frac{\eta}{K}],
 \end{aligned}
\end{eqnarray*}
and in the case of $x_{i2}(t)<L/2$, we have $\theta_i(t+1)\in [\eta/K, 2\eta/K]$ and
 \begin{eqnarray*}\label{lem4_3}
x_{i2}(t+1)\in [x_{i2}(t)+v\sin \frac{\eta}{K},x_{i2}(t)+v\sin \frac{2\eta}{K}].
\end{eqnarray*}
From these and with the condition $\max_{1\leq i\leq n}|\theta_i(0)|\leq \eta/2$ we have
\begin{eqnarray}\label{lem4_1_a}
u_i(t)\in [-\eta+\delta_i(t),\eta-\delta_i(t)],~~\forall 1\leq i\leq n, 0\leq t<t_0,
\end{eqnarray}
and can compute that
 \begin{eqnarray}\label{lem4_4}
\max_{1\leq i\leq n} |\theta_i(t_0)| \leq \frac{2\eta}{K}~\mbox{and}~\max_{1\leq i\leq n}\big|x_{i2}(t_0)-\frac{L}{2}\big| \leq  v\sin \frac{2\eta}{K}.
\end{eqnarray}

Next we proceed with the proof for the following two cases respectively:\\
\textbf{Case I:} $n$ is even or $\varepsilon>1/n$.
 We separate the $n$ agents into two disjoint sets $\mathcal{A}_1$ and $\mathcal{A}_2$ with $|\mathcal{A}_1|=\lceil\frac{n}{2}\rceil$, $|\mathcal{A}_2|=\lfloor \frac{n}{2}\rfloor$,
and $x_{i2}(t)\geq x_{j2}(t)$ for any agent $i\in \mathcal{A}_1, j\in \mathcal{A}_2$.
Let
$$t_1:=t_0+\big\lfloor \frac{r_{\max}}{2v\sin(\frac{\eta}{2}-\frac{\eta}{K})}\big\rfloor+1.$$
For $t_0\leq t<t_1$, we choose
  \begin{eqnarray}\label{lem4_4_a}
u_i(t)=\left\{%
\begin{array}{ll}
\frac{\eta}{2}-\frac{\eta}{2K}-\widetilde{\theta}_i(t)~~~~~~~\mbox{if}~i\in\mathcal{A}_1,\\
-\frac{\eta}{2}+\frac{\eta}{2K}-\widetilde{\theta}_i(t)~~~~\mbox{if}~i\in\mathcal{A}_2.
\end{array}%
\right.
\end{eqnarray}
From this and the protocol (\ref{model2}) we can get
 \begin{eqnarray}\label{lem4_o3_1}
\theta_i(t+1)\in\left\{%
\begin{array}{ll}
 \normalsize[\frac{\eta}{2}-\frac{\eta}{K},\frac{\eta}{2}\normalsize]~~~~~\mbox{if}~i\in\mathcal{A}_1,\\
\normalsize[-\frac{\eta}{2},-\frac{\eta}{2}+\frac{\eta}{K}\normalsize]~~\mbox{if}~i\in\mathcal{A}_2.
\end{array}%
\right.
\end{eqnarray}
Thus, similar to (\ref{lem3_2}) we have
\begin{eqnarray}\label{lem4_o3_2}
\begin{aligned}
x_{i2}(t_1)-x_{j2}(t_1)>r_{\max},~~~~\forall i\in\mathcal{A}_1, j\in\mathcal{A}_2,
\end{aligned}
\end{eqnarray}
and together (\ref{lem4_4}) and (\ref{lem4_o3_1}) we can compute
\begin{eqnarray}\label{lem4_5}
\begin{aligned}
\big|x_{i2}(t_1)-\frac{L}{2}\big|\leq v\sin \frac{2\eta}{K}+ v(t_1-t_0)\sin\frac{\eta}{2}, \forall 1\leq i\leq n.
\end{aligned}
\end{eqnarray}
Also, combining (\ref{lem4_4_a}), (\ref{lem4_o3_1}) and the first inequality of (\ref{lem4_4}) we have
\begin{eqnarray}\label{lem4_5_a}
u_i(t)\in\left[-\eta+\delta_i(t),\eta-\delta_i(t)\right], \forall 1\leq i\leq n, t_0\leq t<t_1.
\end{eqnarray}

Next we will give the control algorithm to minimize the value of the order function.
Set
$$t_2:=\max\left\{t_1+\big\lceil\frac{(\pi-\eta)K+\eta}{2(K-1)\eta} \big\rceil, t_1+1\right\}.$$
For $t_1\leq t<t_2$, we choose
\begin{eqnarray*}\label{lem4_6}
u_i(t)=\left\{%
\begin{array}{ll}
\eta-\frac{\eta}{2K}~~~~~\mbox{if}~\widetilde{\theta}_i(t)<\frac{\pi}{2}+\frac{\eta}{2K}-\eta\\
\frac{\pi}{2}-\widetilde{\theta}_i(t)~~~~\mbox{otherwise}
\end{array}%
\right.
\end{eqnarray*}
for $i\in\mathcal{A}_1$, and
\begin{eqnarray*}\label{lem4_7}
u_i(t)=\left\{%
\begin{array}{ll}
-\eta+\frac{\eta}{2K}~~~~~\mbox{if}~\widetilde{\theta}_i(t)>-\frac{\pi}{2}-\frac{\eta}{2K}+\eta\\
-\frac{\pi}{2}-\widetilde{\theta}_i(t)~~~~\mbox{otherwise}
\end{array}%
\right.
\end{eqnarray*}
 for $i\in\mathcal{A}_2$. From these and the fact of $\max_{1\leq i\leq n}|\theta_i(t_1)|\leq \frac{\eta}{2}$ it can be obtained that
 \begin{eqnarray}\label{lem4_7_a}
u_i(t)\in\left[-\eta+\delta_i(t),\eta-\delta_i(t)\right], \forall 1\leq i\leq n, t_1\leq t<t_2.
\end{eqnarray}
Also, if the sets $\mathcal{A}_1$ and $\mathcal{A}_2$ are disconnected at time $t$, with the similar methods to (\ref{lem1_3}) and (\ref{lem1_5}) we can get
\begin{eqnarray}\label{lem4_8}
&&\min_{i\in\mathcal{A}_1}\theta_i(t+1)\left\{%
\begin{array}{ll}
\in\big[\frac{\pi}{2}-\frac{\eta}{2K},\frac{\pi}{2}+\frac{\eta}{2K}\big]~~\mbox{if}~\min\limits_{i\in\mathcal{A}_1}\theta_i(t)\geq \frac{\pi}{2}+\frac{\eta}{2K}-\eta,\nonumber\\
\geq \frac{(K-1)\eta}{K}+\min\limits_{i\in\mathcal{A}_1}\theta_i(t)~~\mbox{otherwise},
\end{array}%
\right.
\end{eqnarray}
and
\begin{eqnarray*}\label{lem4_9}
&&\max_{i\in\mathcal{A}_2}\theta_i(t+1\left\{%
\begin{array}{ll}
\in\big[-\frac{\pi}{2}-\frac{\eta}{2K},\frac{\eta}{2K}-\frac{\pi}{2}\big]~\mbox{if}~\max\limits_{i\in\mathcal{A}_2}\theta_i(t)\leq \eta-\frac{\eta}{2K}-\frac{\pi}{2},\\
\leq -\frac{(K-1)\eta}{K}+\max\limits_{i\in\mathcal{A}_2}\theta_i(t)~~\mbox{otherwise}.
\end{array}%
\right.
\end{eqnarray*}
Therefore, for any $t_1<t< t_2$, if there exists no edge between $\mathcal{A}_1$ and $\mathcal{A}_2$ at every time in $[t_1,t)$
we can get that:  for all $i\in\mathcal{A}_1$, together (\ref{lem4_o3_1}), (\ref{lem4_8}) and the fact of $u_i(t)+b_i(t)\leq \eta$ we have
\begin{eqnarray*}\label{lem4_9_1}
\begin{aligned}
\frac{\eta}{2}-\frac{\eta}{K}+(t-t_1)(1-\frac{1}{K})\eta\leq \theta_i(t)\leq \frac{\eta}{2}+(t-t_1)\eta,
\end{aligned}
\end{eqnarray*}
 so we can get  $ x_{i2}(t) \geq x_{i2}(t_1)$ and
\begin{eqnarray*}\label{lem4_10}
\begin{aligned}
&x_{i2}(t)=x_{i2}(t_1)+\sum_{t_1<k\leq t}v\sin\theta_i(k)\\
&\leq \frac{L}{2}+v\sin \frac{2\eta}{K}+ v(t_1-t_0)\sin\frac{\eta}{2}\\
&~~~~+v\sum_{k=t_1+1}^{t_2-1}\max_{\alpha\in[-\frac{\eta}{K}+(t-t_1)(1-\frac{1}{K})\eta, (t-t_1)\eta]} \sin\left(\frac{\eta}{2}+\alpha \right)\\
&\leq \frac{L}{2}+\frac{r_{\max}}{2}+v\sum_{k=0}^{\lfloor \frac{\pi}{2\eta}-\frac{1}{2}\rfloor}\sin\left(\frac{\eta}{2}+k\eta \right)~~~~\mbox{as}~K\rightarrow\infty,
\end{aligned}
\end{eqnarray*}
where the first inequality uses (\ref{lem4_5});
symmetrically,
for $j\in\mathcal{A}_2$ we can get
 $ x_{j2}(t) \leq x_{j2}(t_1)$ and
 \begin{eqnarray*}\label{lem4_11}
\begin{aligned}
x_{j2}(t)\geq \frac{L}{2}-\frac{r_{\max}}{2}-v\sum_{k=0}^{\lfloor \frac{\pi}{2\eta}-\frac{1}{2}\rfloor}\sin\left(\frac{\eta}{2}+k\eta \right)~\mbox{as}~K\rightarrow\infty.
\end{aligned}
\end{eqnarray*}
Thus, together these with (\ref{lem4_o3_2}) and the condition $$L>2r_{\max}+2v\sum_{k=0}^{\lfloor \frac{\pi}{2\eta}-\frac{1}{2}\rfloor}\sin\left(\frac{\eta}{2}+k\eta \right),$$ by induction
we can get  $\mathcal{A}_1$ and $\mathcal{A}_2$ are always disconnected during the time interval $[t_1,t_2)$ for large $K$.
Using this and the similar method to (\ref{lem3_6}) we have $\varphi(t_2)\leq\varepsilon$ for large $K$. Combining (\ref{lem4_1_a}), (\ref{lem4_5_a}), (\ref{lem4_7_a}) this yields our result. \\
\textbf{Case II:} $n$ is odd and $\varepsilon\leq\frac{1}{n}$.  We separate the $n$ agents into three disjoint sets $\mathcal{A}_1$, $\mathcal{A}_2$ and $\mathcal{A}_3$ which satisfy that $|\mathcal{A}_1|=|\mathcal{A}_2|=\frac{n-1}{2}, |\mathcal{A}_3|=1$,  and $[X_i(t)]_2\geq [X_j(t)]_2\geq [X_k(t)]_2$ for any agent $i\in \mathcal{A}_1, j\in \mathcal{A}_3$ and $k\in \mathcal{A}_2$.
Let $$t_3:=t_0+\big\lfloor \frac{r_{\max}}{v(\sin(\frac{\eta}{2}-\frac{\eta}{K})-\sin\frac{\eta}{2K})}\big\rfloor+1.$$
For $t_0\leq t<t_3$, we choose $u_i(t)$ to be the same values as (\ref{lem4_4_a}) when $i\in\mathcal{A}_1\cup \mathcal{A}_2$, and to be
$-\widetilde{\theta}_i(t)$ when $i\in\mathcal{A}_3$, which indicates that
 \begin{eqnarray*}\label{lem4_14}
\theta_i(t+1)\in\left\{%
\begin{array}{ll}
 \normalsize[\frac{\eta}{2}-\frac{\eta}{K},\frac{\eta}{2}\normalsize]~~~~~\mbox{if}~i\in\mathcal{A}_1,\\
\normalsize[-\frac{\eta}{2},-\frac{\eta}{2}+\frac{\eta}{K}\normalsize]~~\mbox{if}~i\in\mathcal{A}_2,\\
\normalsize[-\frac{\eta}{2K}, \frac{\eta}{2K}\normalsize]~~\mbox{if}~i\in\mathcal{A}_3.
\end{array}%
\right.
\end{eqnarray*}
Then with the similar argument to (\ref{lem3_2}) we can get
the sets $\mathcal{A}_1$, $\mathcal{A}_2$ and $\mathcal{A}_3$ are mutually disconnected at time $t_3$.

Let $c_n:=\frac{\pi}{2}+\arcsin\frac{1}{n-1}$ and set $$t_4:=\max\left\{t_3+\big\lceil\frac{(c_n-\frac{\eta}{2})K+\frac{\eta}{2}}{(K-1)\eta} \big\rceil, t_3+1\right\}.$$
For $t_3\leq t<t_4$,  we choose
\begin{eqnarray*}\label{lem4_15}
u_i(t)=\left\{%
\begin{array}{ll}
\eta-\frac{\eta}{2K}~~~~~\mbox{if}~\widetilde{\theta}_i(t)<c_n+\frac{\eta}{2K}-\eta\\
c_n-\widetilde{\theta}_i(t)\big)~~~~\mbox{otherwise}
\end{array}%
\right.
\end{eqnarray*}
for $i\in\mathcal{A}_1$, and
\begin{eqnarray*}\label{lem4_16}
u_i(t)=\left\{%
\begin{array}{ll}
-\eta+\frac{\eta}{2K}~~~~~\mbox{if}~\widetilde{\theta}_i(t)>-c_n-\frac{\eta}{2K}+\eta\\
-c_n-\widetilde{\theta}_i(t)~~~~\mbox{otherwise}
\end{array}%
\right.
\end{eqnarray*}
 for $i\in\mathcal{A}_2$, and $u_i(t)=-\widetilde{\theta}_i(t)$ for $i\in\mathcal{A}_3$.
With the similar argument to Case I and using the condition of
  $$L>2r_{\max}+2v\sum_{k=0}^{\lfloor \frac{c_n}{\eta}-\frac{1}{2}\rfloor}\sin\left(\frac{\eta}{2}+k\eta \right)$$
  we can get $\mathcal{A}_1$, $\mathcal{A}_2$ and $\mathcal{A}_3$ are mutually disconnected at every time from $t_3$ to $t_4$, and so $\varphi(t_4)\leq\varepsilon$ by the similar method of (\ref{lem3_17}).
Also, similar to Case I we can get $$u_i(t)\in\left[-\eta+\delta_i(t),\eta-\delta_i(t)\right], ~~\forall 1\leq i\leq n, t_0\leq t<t_4.$$
Together these with (\ref{lem4_1_a}) our result is obtained.

For the protocol (\ref{model2_new}), if $\eta<\pi/2$ we can get our result with the similar method as protocol (\ref{model2}). Otherwise,
by Lemma \ref{lem1_2} we can control the state of the system to $S_{\eta'}^1$ with $\eta'<\pi/2$, then with the similar method as protocol (\ref{model2}) yields our result.

\section{Proof of Lemma \ref{lem4_new}}\label{App_lem4_new}

We consider  protocol (\ref{model2_new}) first. Let $b$ be the middle value of the minimal interval contains all the initial headings of the agents. Without loss of generality we assume $b\in [0,\pi/4)$.
Let $t_0:=\lceil \frac{\pi}{\eta}\rceil$-1. For $t\in [0,t_0)$ and $1\leq i \leq n$, we choose
\begin{eqnarray*}\label{lem4_new_1}
\left(\delta_i(t),u_i(t)\right)=\left\{%
\begin{array}{ll}
(\eta/4,-3\eta/4)~~~~\mbox{if}~\widetilde{\theta}_i(t)>b+\eta/2,\\
(\eta/2,b-\widetilde{\theta}_i(t))~\mbox{if}~\widetilde{\theta}_i(t)\in [b-\eta/2,b+\eta/2],\\
(\eta/4,3\eta/4)~~~~~~~\mbox{if}~\widetilde{\theta}_i(t)<b-\eta/2.
\end{array}%
\right.
\end{eqnarray*}
Similar to (\ref{lem3_new_1}) we can get
\begin{eqnarray*}\label{lem4_new_2}
\max_{1\leq i \leq n}|\theta_i(t_0)-b|\leq \frac{\eta}{2}.
\end{eqnarray*}
Set $t_1:=t_0+\lceil \frac{\pi}{2\eta}\rceil$-1. For $t\in [t_0,t_1)$ and $1\leq i \leq n$, we choose
\begin{eqnarray*}\label{lem4_new_3}
\left(\delta_i(t),u_i(t)\right)=\left\{%
\begin{array}{ll}
(\eta/4,-3\eta/4)~~\mbox{if}~\widetilde{\theta}_i(t)>\eta/2,\\
(\eta/2,-\widetilde{\theta}_i(t))~~\mbox{if}~\widetilde{\theta}_i(t)\in [-\eta/2,\eta/2].
\end{array}%
\right.
\end{eqnarray*}
With the similar method to (\ref{lem3_new_1}) again we have
\begin{eqnarray}\label{lem4_new_4}
\max_{1\leq i \leq n}|\theta_i(t_1)|\leq \frac{\eta}{2}.
\end{eqnarray}
Set $t_2:=t_1+\lceil \frac{L}{2v\sin \frac{\eta}{K}} \rceil$.
For $i=1,\ldots,n$ and $t\in[t_1,t_2)$, we choose $\delta_i(t)=\frac{\eta}{2K}$ and $u_i(t)$ as (\ref{lem4_1}).
Similar to (\ref{lem4_4}) we have
 \begin{eqnarray*}\label{lem4_new_5}
\max_{1\leq i\leq n} |\theta_i(t_2)| \leq \frac{2\eta}{K}~\mbox{and}~\max_{1\leq i\leq n}\big|x_{i2}(t_2)-\frac{L}{2}\big| \leq  v\sin \frac{2\eta}{K}.
\end{eqnarray*}

Next we separate the $n$ agents into four disjoint nonempty sets $\mathcal{A}_i$, $i=1,2,3,4$, which satisfy that $[X_i(t)]_2\geq [X_j(t)]_2\geq [X_k(t)]_2\geq [X_k(t)]_2$ for any agent $i\in \mathcal{A}_1, j\in \mathcal{A}_2, k\in \mathcal{A}_3$ and $l\in\mathcal{A}_4$.
Let
$$t_3:=t_2+\big\lfloor \frac{r_{\max}}{2v\sin(\frac{\eta}{2}-\frac{\eta}{K})}\big\rfloor+1.$$
For $t_2\leq t<t_3$, we choose
  \begin{eqnarray*}\label{lem4_new_7}
u_i(t)=\left\{%
\begin{array}{ll}
\frac{\eta}{2}-\frac{\eta}{2K}-\widetilde{\theta}_i(t)~~~~~~~\mbox{if}~i\in\mathcal{A}_1\cup\mathcal{A}_2,\\
-\frac{\eta}{2}+\frac{\eta}{2K}-\widetilde{\theta}_i(t)~~~~\mbox{if}~i\in\mathcal{A}_3\cup\mathcal{A}_4.
\end{array}%
\right.
\end{eqnarray*}
Set
$$t_4:=\max\left\{t_3+\big\lceil\frac{(\pi-\eta)K+2\eta}{2(K-1)\eta} \big\rceil, t_3+1\right\}.$$
In the time $t\in [t_3,t_4)$, we choose
\begin{eqnarray*}\label{lem4_new_8}
u_i(t)=\left\{%
\begin{array}{ll}
\eta-\frac{\eta}{2K}~~~~~\mbox{if}~\widetilde{\theta}_i(t)<\frac{\pi}{2}+\frac{\eta}{K}-\eta\\
\frac{\pi}{2}+\frac{\eta}{2K}-\widetilde{\theta}_i(t)~~~~\mbox{otherwise}
\end{array}%
\right.
\end{eqnarray*}
for $i\in\mathcal{A}_1$, and
\begin{eqnarray*}\label{lem4_new_9}
u_i(t)=\left\{%
\begin{array}{ll}
\eta-\frac{\eta}{2K}~~~~~\mbox{if}~\widetilde{\theta}_i(t)<\frac{\pi}{2}-\eta\\
\frac{\pi}{2}-\frac{\eta}{2K}-\widetilde{\theta}_i(t)~~~~\mbox{otherwise}
\end{array}%
\right.
\end{eqnarray*}
for $i\in\mathcal{A}_2$, and
\begin{eqnarray*}\label{lem4_new_10}
u_i(t)=\left\{%
\begin{array}{ll}
-\eta+\frac{\eta}{2K}~~~~~\mbox{if}~\widetilde{\theta}_i(t)>-\frac{\pi}{2}+\eta\\
-\frac{\pi}{2}+\frac{\eta}{2K}-\widetilde{\theta}_i(t)~~~~\mbox{otherwise}
\end{array}%
\right.
\end{eqnarray*}
 for $i\in\mathcal{A}_3$, and
 \begin{eqnarray*}\label{lem4_new_11}
u_i(t)=\left\{%
\begin{array}{ll}
-\eta+\frac{\eta}{2K}~~~~~\mbox{if}~\widetilde{\theta}_i(t)>-\frac{\pi}{2}-\frac{\eta}{K}+\eta\\
-\frac{\pi}{2}-\frac{\eta}{2K}-\widetilde{\theta}_i(t)~~~~\mbox{otherwise}
\end{array}%
\right.
\end{eqnarray*}
 for $i\in\mathcal{A}_4$. With the similar discuss as the Case I of the proof of Lemma \ref{lem4} we can get $d_{\theta(t_4)}\geq\pi$
under condition  (\ref{theo2_org_00}).

For protocol (\ref{model2}), combining (\ref{lem4_4}) and the process of the above paragraph our result follows.

\section{Proofs of Theorems \ref{result_4}, \ref{turn_2}, \ref{vortex_2} and \ref{relt_assmp}}\label{App_result_4}

\begin{proof} \textbf{of Theorem \ref{result_4}}
Given $t>0$,  suppose $\max_{1\leq i\leq n}|\theta_i(t)|\leq \frac{\eta}{2}$. We separate the $n$ agents into two disjoint sets $\mathcal{A}_1$ and $\mathcal{A}_2$ with $|\mathcal{A}_1|=\lceil\frac{n}{2}\rceil$, $|\mathcal{A}_2|=\lfloor\frac{n}{2}\rfloor$,
and $x_{i2}(t)\geq x_{j2}(t)$ for any agent $i\in \mathcal{A}_1, j\in \mathcal{A}_2$. Set
$T_1:=\big\lfloor \frac{r_{\max}}{2v\sin(\eta/4)}\big\rfloor+1$, where $K$ is an integer not smaller than $4$,
and set $T$ be an arbitrary large integer. Under protocol (\ref{model2}) (or (\ref{model2_new})), for $t'\in[t,t+T_1+T)$ we choose $\delta_i(t')=\frac{\eta}{2K}$ for $1\leq i\leq n$,
\begin{eqnarray*}
u_i(t')=\left\{%
\begin{array}{ll}
-\frac{3\eta}{2K}-\widetilde{\theta}_i(t')~~~~~\mbox{if}~x_{i2}(t')\geq \frac{3L}{4}\\
\frac{3\eta}{2K}-\widetilde{\theta}_i(t')~~~~~~~\mbox{if}~x_{i2}(t')< \frac{3L}{4},
\end{array}%
\right.
\end{eqnarray*}
for $i\in\mathcal{A}_1$, and
\begin{eqnarray*}
u_i(t')=\left\{%
\begin{array}{ll}
-\frac{3\eta}{2K}-\widetilde{\theta}_i(t')~~~~~\mbox{if}~x_{i2}(t')\geq L/4\\
\frac{3\eta}{2K}-\widetilde{\theta}_i(t')~~~~~~\mbox{if}~x_{i2}(t')< L/4
\end{array}%
\right.
\end{eqnarray*}
for $i\in\mathcal{A}_2$. Similar to
(\ref{lem4_4}), for all $t'\in [t+T_1,t+T_1+T]$  we can get,
 \begin{eqnarray*}
\max_{i\in\mathcal{A}_1}\big|x_{i2}(t')-\frac{3L}{4}\big| \leq  v\sin \frac{2\eta}{K}
\end{eqnarray*}
and
 \begin{eqnarray*}
\max_{ i\in \mathcal{A}_2}\big|x_{i2}(t')-\frac{L}{2}\big| \leq  v\sin \frac{2\eta}{K},
\end{eqnarray*}
which indicates that if $L>2r_{\max}$ then $\cup_{t'=t_0}^{t_0+T}G(t')$ is not weakly connected for large $K$.
Under protocol (\ref{model1}) (or (\ref{m1_new})), similar to (\ref{result_2_1}) we have
\begin{eqnarray*}
&&P\left(\bigcup_{t'=t+T_1}^{t+T_1+T}G(t') \mbox{ is not weakly connected}\big| \forall (X(t),\theta(t))\in S_{\eta}^1 \right)\geq (2K)^{-n(T_1+T)}.
\end{eqnarray*}
Because for protocol (\ref{model2}) (or (\ref{model2_new}) with $\eta>\frac{\pi}{2}-\frac{\pi}{n}$) $S_{\eta}^1$ is also finite-time robustly reachable from $ S^*$ under the periodic boundary conditions,
with the similar process from (\ref{result_2_2}) to the end of the proof of Theorem \ref{result_2} we can get our result.
\end{proof}

\vskip 2mm

\begin{proof}\textbf{ of Theorem \ref{turn_2}}
With the same discussion to the first paragraph of the proof of Theorem \ref{turn_1} we can get the event of turn
will happen an infinite number of times with probability $1$.

Given a time $t_1$, suppose $\max_{1\leq i\leq n}|\theta_i(t_1)|\leq \frac{\varepsilon}{2}$ for a small constant $\varepsilon>0$. Under
the similar process from (\ref{lem4_new_4}) to the end of the proof of Lemma \ref{lem4} we can get the event of bifurcation happens
in the time $[t_1,t_4)$, where $t_4$ is the same constant in the proof of Lemma \ref{lem4}. Also, with the similar process to the proof
of Lemma \ref{lem1} (or \ref{lem1_2}) we can get there exist a time $t_5>t_4$ such that the event of merging happens in the time $[t_4,t_5)$.
 Using Lemmas \ref{robust} and \ref{lem1} (or \ref{lem1_2}) we can get the events of bifurcation and merging will happen an infinite number of times with probability $1$.
\end{proof}

\vskip 2mm
\begin{proof} \textbf{of Theorem \ref{vortex_2}}
This proof is as same as the proof of Theorem \ref{vortex_1} but using Theorem \ref{turn_2} instead of Theorem \ref{turn_1}.
\end{proof}
\vskip 2mm

\begin{proof} \textbf{of Theorem \ref{relt_assmp}}
By our assumptions we can get $\Omega_{\pi}^3$ is finite-time robustly reachable from $\Omega^*$ under protocol (\ref{model2_new}).
Also, for any $\alpha>0$,  by (\ref{lem3_new_1}) we can get $\Omega_{\alpha}^1$ is finite-time robustly reachable from $\Omega_{\pi}^3$ under protocol  (\ref{model2_new}). Thus, similar to Lemma \ref{lem1} we have $\Omega_{\alpha}^1$ is finite-time robustly reachable from $\Omega^*$ under protocol (\ref{model2_new}). With the same proofs of
Theorems \ref{result_1}, \ref{result_2}, \ref{turn_1}, \ref{vortex_1}, \ref{result_3}, \ref{result_4}, \ref{turn_2} and \ref{vortex_2} but using
this instead of Lemma \ref{lem1_2} our results can be obtained.
 \end{proof}

\end{document}